\documentclass[11pt,a4paper]{amsart}
\usepackage{a4wide}
\usepackage{amsfonts,amsthm,amsmath,amssymb}
\usepackage{amscd,color}
\usepackage{verbatim}
\usepackage{graphicx,subfigure}
\usepackage{pgfplots} 
\usepackage{mathtools}
\usepackage{marginnote}
\usepackage{soul}

\newtheorem*{thmA}{Theorem A}

\theoremstyle{plain}

\newtheorem{theorem}{Theorem}[section]
\newtheorem{lemma}[theorem]{Lemma}
\newtheorem{proposition}[theorem]{Proposition}
\newtheorem{corollary}[theorem]{Corollary}

\newtheorem{remark}[theorem]{Remark}

\newcommand{\R}{\mathbb{R}}

\newcommand{\A}{\mathcal{A}}

\newcommand{\D}{\mathcal{D}}

\newcommand{\TT}{\mathcal{T}}

\newcommand{\K}{\mathcal{K}}

\newcommand{\graph}{\mathrm{graph}\,}
\newcommand{\image}{\mathrm{image}\,}

\newcommand{\Int}{\mathrm{int}\,}
\newcommand{\loc}{\mathrm{loc}\,}

\newcommand{\cal}{\mathcal}

\newcommand{\wh}{\widehat}
\newcommand{\wt}{\widetilde}

\newcommand{\J}{\ensuremath{{\cal J}}}

\newcommand{\al}{\alpha}
\newcommand{\be}{\beta}
\newcommand{\ga}{\gamma}

\newcommand{\la}{\lambda}
\newcommand{\ph}{\varphi}

\renewcommand{\r}{\color{red}}

\newsavebox{\savepar}

\numberwithin{equation}{section}

\graphicspath{{./figures/}}

\begin{document}

\title[A model for the dynamics of the secant map near a critical three-cycle] {
On the basin of attraction of a critical three-cycle of a model for the secant map
}

\date{\today}

\author{Ernest Fontich}
\address{Departament de Matem\`atiques i Inform\`atica, Universitat de Barcelona (UB), 
Gran Via de les Corts Catalanes 585, 08007 Barcelona, Catalonia,  Spain
and 
Centre de Recerca Matem\`atica (CRM), 
08193 Bellaterra, Barcelona, Catalonia, Spain}
\email{fontich@ub.edu}
\author{Antonio Garijo}
\address{Departament d'Enginyeria Inform\`atica i Matem\`atiques,
Universitat Rovira i Virgili, 43007 Tarragona, Catalonia, Spain}
\email{antonio.garijo@urv.cat}
\author{Xavier Jarque}
\address{Departament de Matemàtiques i Informàtica, Universitat de Barcelona, Gran Via, 585, 08007 Barcelona, Catalonia and  Centre de Recerca Matemàtica, Edifici C, Campus Bellaterra, 08193 Bellaterra, Catalonia}
\email{xavier.jarque@ub.edu}

\thanks{The first author is supported by the
	grant PID2021-125535NB-I00.   
	The second and  third  authors were  supported by MICIU/AEI grant PID2020-118281GB-C32(33) Both grant are funded by MICIU/AEI/10.13039/501100011033/FEDER,UE.  The second author is supported by  Generalitat de Catalunya  2021SGR-633. We want to thank  the Thematic Research Programme {\it Modern holomorphic dynamics and related fields}, Excellence Initiative – Research University programme at the University of Warsaw.
	Finally, this work has also been funded through the Mar\'ia de Maeztu Program for Centers and Units of Excellence in R\&D, grant CEX2020-001084-M funded by MICIU/AEI/10.13039/501100011033/FEDER,UE}

\begin{abstract}
 We consider the secant method $S_p$ applied to a   real  polynomial $p$ of degree $d+1$ as a  discrete dynamical system on $\mathbb R^2$. If the polynomial $p$ has a local extremum at a point $\alpha$ then the discrete dynamical system generated by the iterates of the secant map exhibits a
 critical periodic orbit of period 3 or three-cycle at the point $(\alpha,\alpha)$. We propose a simple model map $T_{a,d}$ having a unique fixed point at the origin which encodes the dynamical behaviour of $S_p^3$ at the critical  three-cycle. The main goal of the paper is to describe the geometry and topology of the basin of attraction of the origin of $T_{a,d}$ as well as its boundary. Our results concern global, rather than local, dynamical behaviour. They include that the boundary of the basin of attraction is the stable manifold of a fixed point or contains the stable manifold of a  two-cycle, depending on the  values of the parameters of $d$  (even or odd) and $a\in \mathbb R$ (positive or negative). 
\end{abstract}

\maketitle

%
%
%
%
%
%

\section{ Introduction}

A major goal in applied and theoretical mathematical modelling is to find  stable equilibria which determines the expected behaviour of the phenomenon we are analyzing. Those equilibria are given by real (or complex) numbers, real (or complex) finite dimensional vectors, or functions belonging to an infinite dimensional space, depending on the nature of the model under consideration. 

In the majority of cases, the stable equilibria determining the evolutionary steady states of any model turn out to be solutions of non-linear equations. In general, we cannot solve these equations explicitly.  Accordingly, there is a long history of research studying different {\it algorithms} which efficiently find their solutions.

Among these algorithms the ones given by the special kind of discrete dynamical systems, known as 
root-finding algorithms, have been shown to be the most useful. Roughly speaking, a {\it root-finding algorithm} is a  system such that for most of the initial conditions the asymptotic behaviour of the corresponding iterative process tends to one of the solutions of the non-linear equation determining the equilibria of the model. We observe that the condition of convergence {\it for most of the initial conditions} is a global phenomenon rather than (only) a local one. Indeed, this is the reason why the {\it global dynamics} of root-finding algorithms has been an important subject of study for  dynamicists. 

Moreover, when the model has more than one steady state the phase portrait of the root-finding algorithm splits into regions where the iterates of the seeds converge to different equilibria. Consequently, two natural questions arise. On the one hand, about the boundaries of these regions: do they have easy geometry and topology? what about the restricted dynamics over these boundaries? 
 do they have positive measure?
On the other hand, about the stable steady states: are there other stable behaviour of the algorithm unrelated to the steady states of the model? If they exist, there would be open sets of the dynamical plane where the root-finding algorithm is full of {\it bad} initial conditions. The answers to all these questions has had a great influence on the study of theoretical as well as applied discrete dynamical systems.

There is no discussion that the most famous root-finding algorithm is the well-known Newton's method. More concretely, assume the equation we want to solve is $p(x)=0$, where, to simplify the exposition, we assume $p(x)$ to be a  polynomial with $x\in \mathbb R$ or $x\in \mathbb C$ but the method extends to higher dimensional problems.  {\it Newton's method} is the study of the dynamical system
\begin{equation}\label{eq:newton}
x_{n+1}=N_p\left(x_n\right):= x_n-\frac{p(x_n)}{p^\prime(x_n)},\qquad  x_0 \in \mathbb R \quad \text{or} \quad x_0 \in \mathbb C.
\end{equation}
It is easy to see that $N_p(\alpha)=\alpha$ if and only if $p(\alpha)=0$ (if $p^{(\ell)}(\alpha)=0,\ 1\le \ell \le k$, one can still use \eqref{eq:newton} after some modifications) and moreover if $x_0\approx \alpha$ then $\{x_n:=N^n_p(x_0)\} \to \alpha$ as $n\to \infty$. Consequently, Newton's method is a (one dimensional) dynamical systems whose fixed points correspond to the roots of $p$ and they are  local attractors. In fact, it is somehow {\it the canonical} roof-finding algorithm. The natural dynamical space for one dimensional Newton's method  applied to degree $d$ polynomials is $\mathbb C$ (rather than $\mathbb R$) due to the fundamental theorem of algebra, and it defines one of the most studied family of rational (holomorphic) maps on the Riemann sphere \cite{Bla1,Bla2,Shi,HowToNewton}. See also \cite{Ber,BarFagJarKar} for Newton's method applied to transcendental entire maps. 

Despite its fundamental role, Newton's method has some limitations and the literature has explored other root finding algorithms trying to avoid these weakness (for instance avoiding to compute the derivatives if their computational cost is too high) or to improve the efficiency of the method under certain hypothesis (for instance improving the local speed of convergence of the method to the root(s) of $p$).

Another basic 
root-finding algorithm is the {\it secant method} given by the dynamical system generated by the iterates of the 2-dimensional map
\begin{equation}\label{eq:secant}
S(x,y)=S_p(x,y):=\left(y,y-p(y) \frac{x-y}{p(x)-p(y)}\right).
\end{equation}
However, in contrast to Newton's method, secant's method is a two dimensional system and it does not required to compute any derivative of $p$. Nonetheless, as before, we have that $S(\alpha,\alpha)=(\alpha,\alpha)$ if and only if $p(\alpha)=0$. Moreover if $(x_0,y_0)\approx (\alpha,\alpha)$ then 
$$
\{\left(x_n,y_n\right):=S^n(x_0,y_0)\} \to (\alpha,\alpha)
$$
at least for simple roots of $p$ (see \cite{Multiple} for multiple roots). We refer to \cite{Tangent, BedFri}, and references therein, for a detailed discussion of the phase plane ($\mathbb R^2$ and $\mathbb C^2$) of the secant method.  

One (unexpected) fact of the real secant map is that there are no  finite periodic points of period two or three in $\mathbb R^2$. However, it has \textit{finite} periodic points of period four and some of them determine the geometry and topology of the boundaries of the immediate basin of attraction of its fixed points (see \cite{GarGarJar}). Moreover, if we extend the domain of the secant method to {\it infinity} (that is, if we extend the secant map to $\mathbb {RP}^2$ or $\mathbb {CP}^2$) a new three-cycle phenomenon arises. Indeed, in \cite{BedFri} (see also \cite{Tangent}) the authors showed that if $c\in \mathbb R$ satisfies that $p'(c)=0$ (critical point) and $p(c) p''(c)\neq 0$ the secant method exhibits a {\it critical} three-cycle at $(c,c)$ 
given by
\[
(c,c)  \xmapsto{\text{S}}   (c,\infty) \xmapsto{\text{S}}  (\infty,c) \xmapsto{\text{S}} (c,c).
\]
Moreover, the three-cycle has a basin of attraction whose geometry varies depending on the degree of the polynomial. However,  its geometry and topology is quite similar among polynomials of the same degree. These basins, and their 
disparate geometry can be visualized in red in Figure \ref{fig:secant} for concrete 
polynomials of degree of different parity. 

The main goal of this work is to go deeper into the understanding of the geometry of the basin of the critical three-cycle by means of a model which captures the relevant information and allows us to give an accurate description.

\begin{figure}[hbt!]
	\centering
	\begin{tikzpicture}
		\begin{axis}[width=200pt, axis equal image, scale only axis,  enlargelimits=false, axis on top]
			\addplot graphics[xmin=-1.5,xmax=1.5,ymin=-1.5,ymax=1.5] {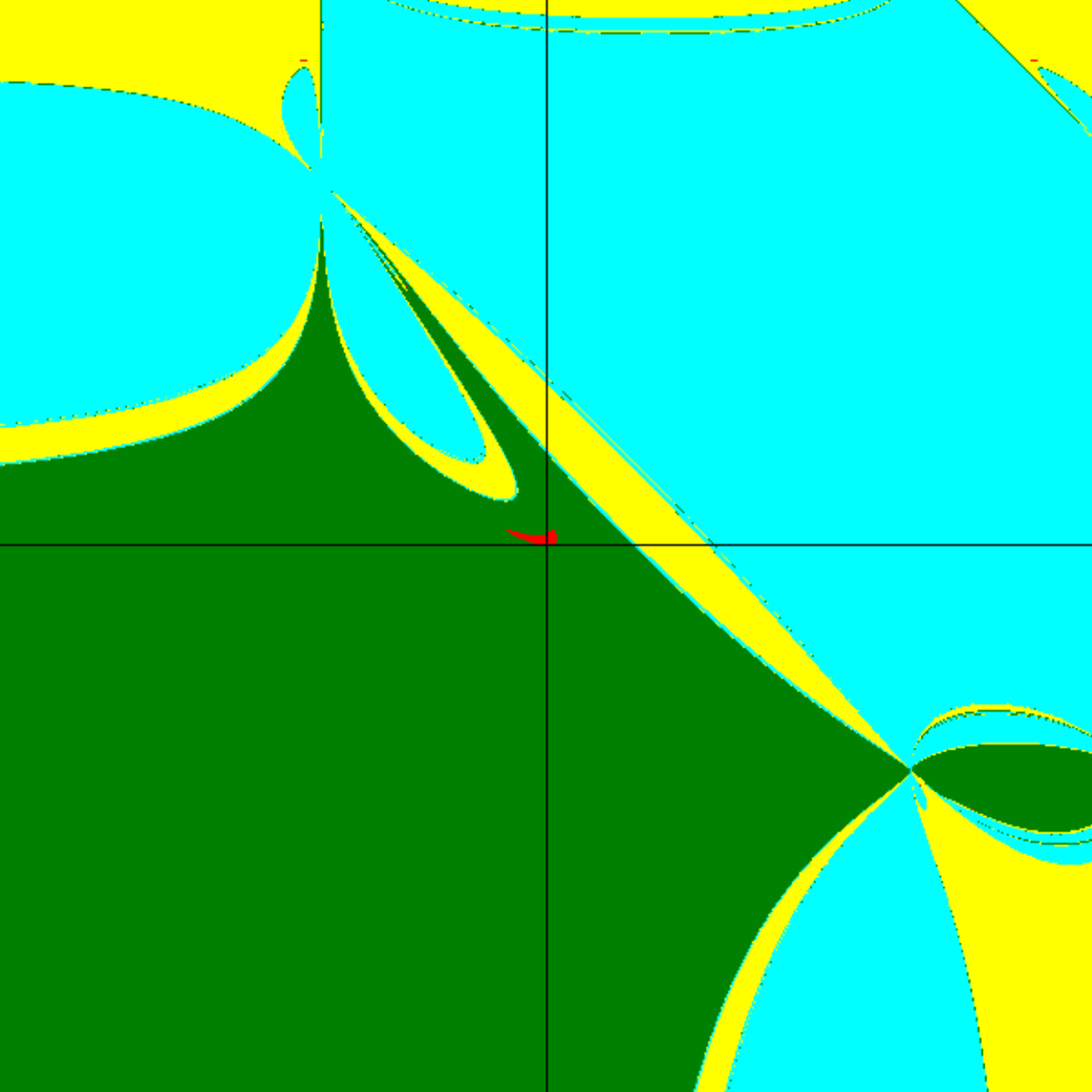};
		\end{axis}
	\end{tikzpicture}
	\begin{tikzpicture}
		\begin{axis}[width=200pt, axis equal image, scale only axis,  enlargelimits=false, axis on top]
			\addplot graphics[xmin=-0.2,xmax=0.2,ymin=-0.2,ymax=0.2] {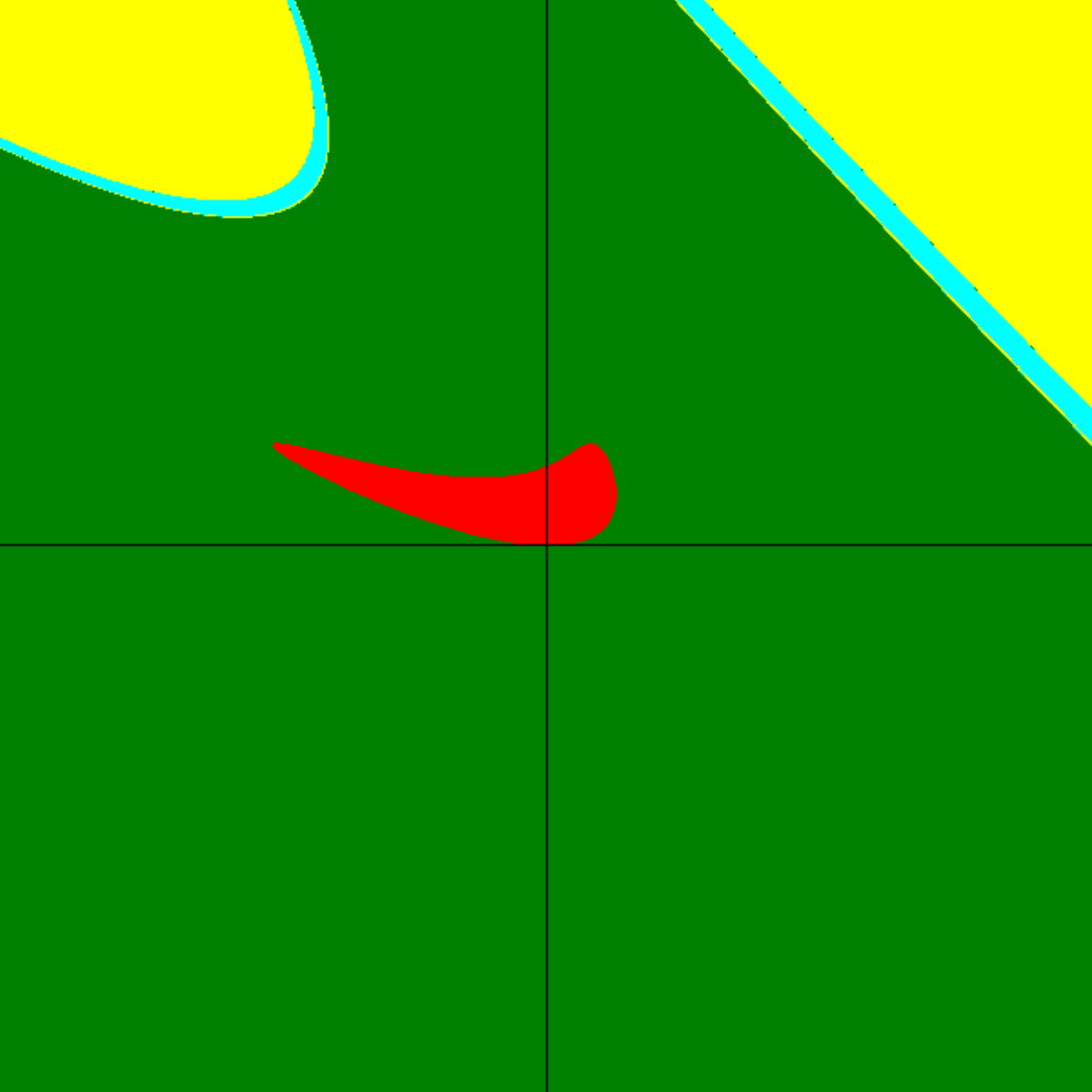};
		\end{axis}
	\end{tikzpicture}
	\begin{tikzpicture}
		\begin{axis}[width=200pt, axis equal image, scale only axis,  enlargelimits=false, axis on top]
			\addplot graphics[xmin=-1.5,xmax=1.5,ymin=-1.5,ymax=1.5] {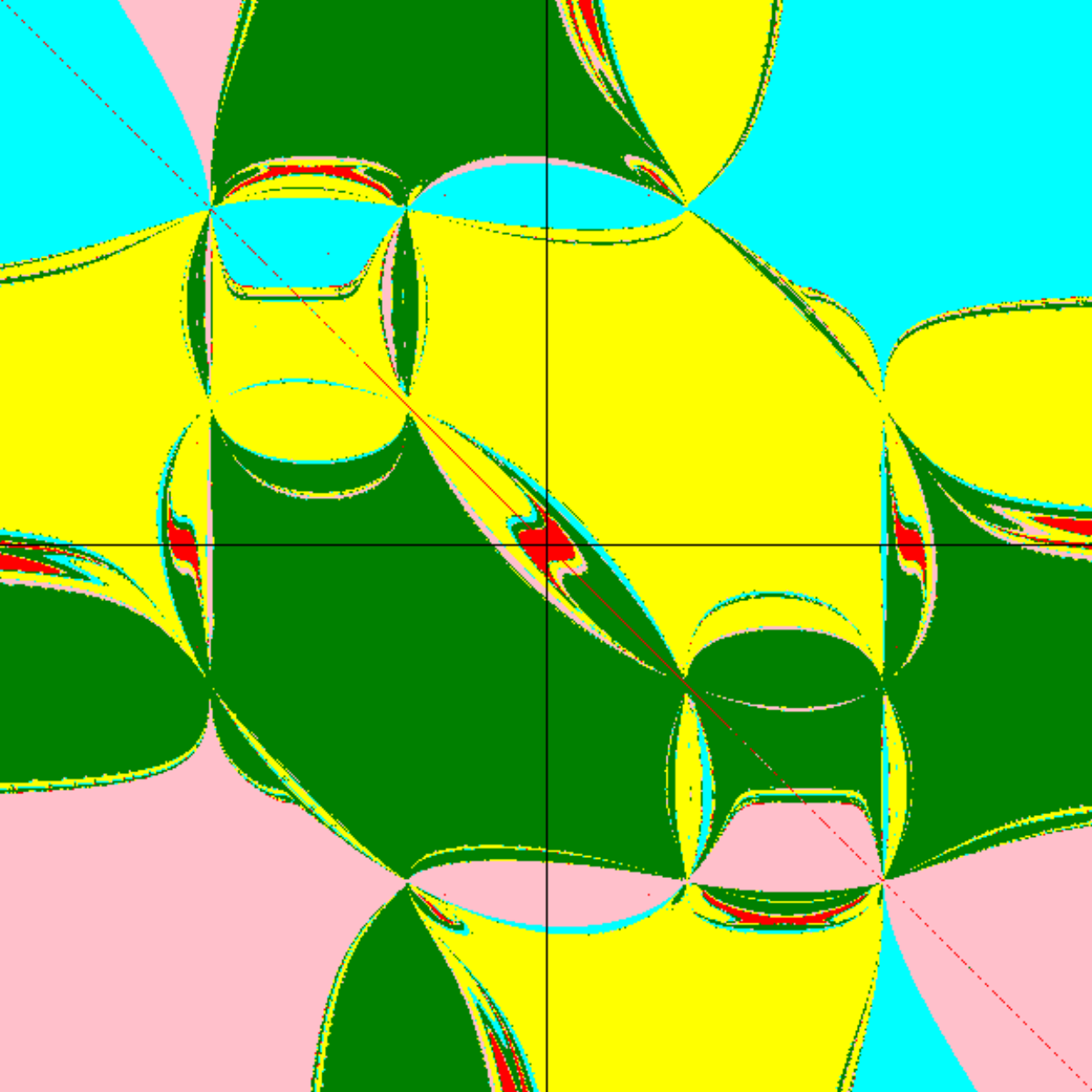};
		\end{axis}
	\end{tikzpicture}
	\begin{tikzpicture}
		\begin{axis}[width=200pt, axis equal image, scale only axis,  enlargelimits=false, axis on top]
			\addplot graphics[xmin=-0.5,xmax=0.5,ymin=-0.5,ymax=0.5] {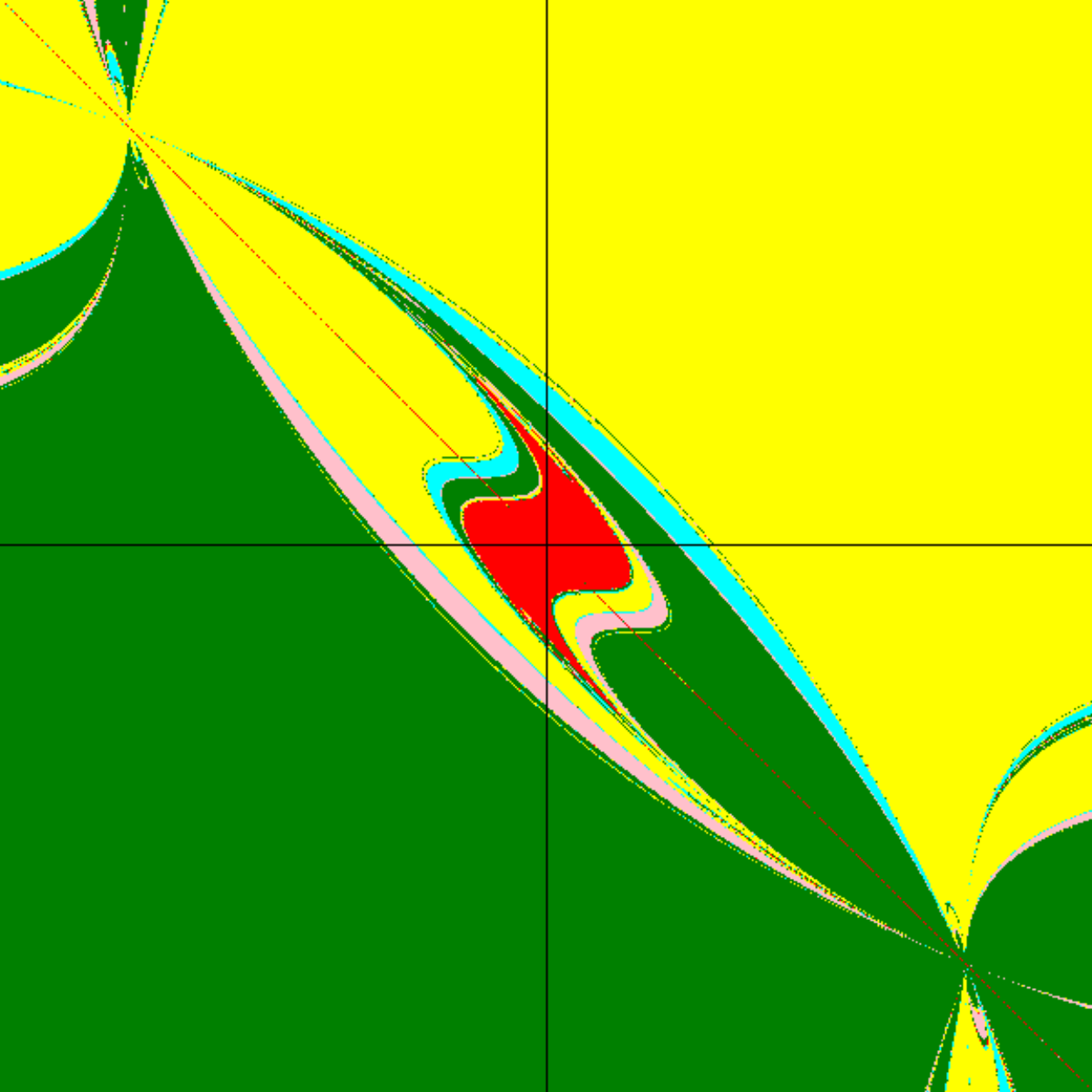};
		\end{axis}
	\end{tikzpicture}

	\caption{Phase planes of the secant map applied to the polynomials  $p(x)= 1 -2 x^2 + x^3$ (first row) and $p(x)=1-8x^2+8x^4$ (second row). In all the pictures we show in red  the set of points converging towards the  critical three-cycle $\{(0,0),(0,\infty),(\infty,0)\}$. The second column is a zoom  near the origin of the first one. } 
	\label{fig:secant}
\end{figure}

Following the approach in \cite{BedFri} we assume, without lost of generality, that $c=0$ and $p(0)=1$. Thus, assuming also that $\mbox{deg}(p)=d+1$, the polynomial $p$ writes as 
\begin{equation}\label{eq:p}
p(x)= 1 + a_2 x^2 + \ldots + a_{d+1}x^{d+1},
\end{equation}
\noindent where $d \geq 2$ and $a_2a_{d+1} \neq 0$. Using the natural extension of $S$ at infinity, via  the charts $\varphi_1(x,y)=(1/x,y)$  and $\varphi_2(x,y)=(x,1/y)$, and some computations (explicit in \cite{BedFri}) 
the expression of $S^3$ near the origin is given by
  \begin{equation}\label{eq:S3}
S^3
\left(
\begin{array}{l}
 x   \\ 
 y 
\end{array}
\right)  = 
 \left(
\begin{array}{l}
y - \frac{(-a_2)^{d}}{a_{d+1}} (x+y)^{d} \\ 
y - 2 \frac{(-a_2)^{d}}{a_{d+1}} (x+y)^{d}
\end{array}
\right)
+ \mathcal O_{d+1},
\end{equation}
where $\mathcal O_{d+1}$ indicates terms bounded by $(|x|+|y|)^{d+1}$.  The expression \eqref{eq:S3} motivates the introduction of the following model map $T_{a,d}$ which encodes the dominant terms of $S^3$ near the origin. Concretely,
\begin{equation}\label{eq:T}
T_{a,d}\left(
\begin{array}{l}
 x   \\ 
 y 
\end{array}
\right)  = 
 \left(
\begin{array}{l}
y - a (x+y)^d \\ 
y - 2a(x+y)^d
\end{array}
\right), 
\end{equation}
where $a\neq 0$ is a parameter. 
We now are ready to state the main results of this paper about the basin of attraction of the origin of \eqref{eq:T}, defined as
\begin{equation}\label{eq:basin}
\mathcal A_{a,d}(0) = \{ (x,y) \in \mathbb R^2\, | \ T_{a,d}^n(x,y) \to (0,0) \,  \hbox{ as }  \, n \to \infty  \} ,
\end{equation}
depending on the parameters $a$ and $d$. Obviously, the origin is the unique fixed point of $T_{a,d}$ and $DT_d(0,0)$ has eigenvalues $0$ and $1$. The $0$ eigenvalue guarantees that $\mathcal A_{a,d}(0)\ne \emptyset$ but a complete topological and geometric description depends on the motion over the center manifold. The main theorem describes $\mathcal A_{a,d}(0)$ as well as its boundary $\partial \mathcal A_{a,d}(0)$ depending on $a$ and $d$. 

 \begin{thmA}
 Let $\mathcal A_{a,d}(0)$ be the basin of attraction of the origin for the map $T_{a,d}$. 
  \begin{enumerate}
 \item[(a)]If $d$ is even and $a \neq 0$ then  $ \mathcal A_{a,d}(0)$  is a compact  set which is homeomorphic to a closed topological disk and the  boundary of $\mathcal A_{a,d}(0)$ is the stable manifold  of the origin. See Figure \ref {fig:dynamical_plane}(a).
 \item[(b)] If $d$ is odd and $a>0$ then $ \mathcal A_{a,d}(0)$ is an open, simply connected, unbounded  set. Moreover, $\partial  \mathcal A_{a,d}(0)$ contains the stable manifold of a hyperbolic two-cycle  $\{p_0,p_1\}$ lying on   $\partial  \mathcal A_{a,d}(0)$. See Figure \ref {fig:dynamical_plane}(b).
 \item[(c)] If $d$ is odd and $a<0$ then  $ \mathcal A_{a,d}(0)$ is the stable manifold of the origin. Moreover,  $ \mathcal A_{a,d}(0)$ is unbounded. See Figure \ref {fig:dynamical_plane}(c).
 \end{enumerate}
\end{thmA}

\begin{figure}[ht]
    \centering
    \subfigure[\scriptsize{  $d =2$ and $a=1$. }]{
     \includegraphics[width=0.3\textwidth]{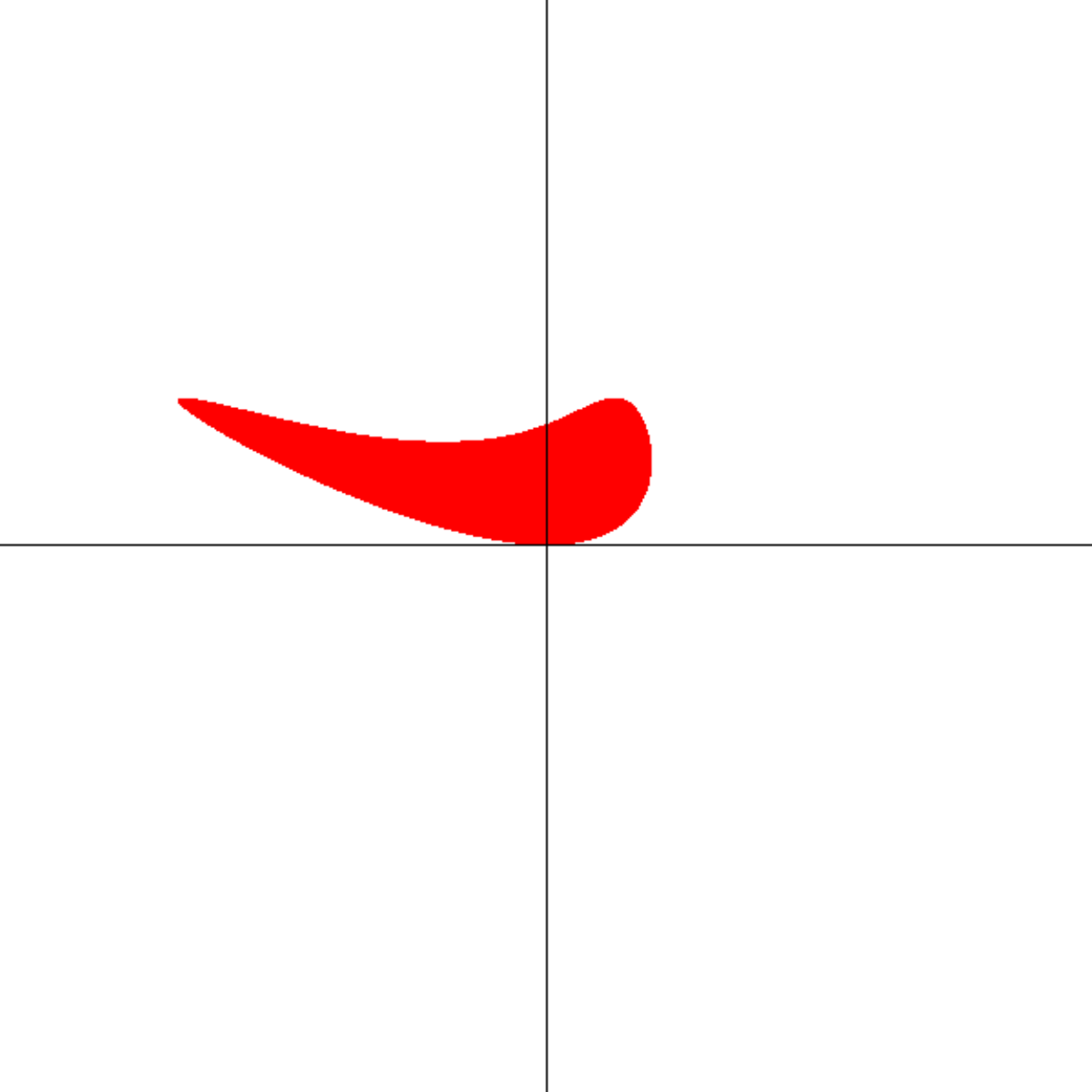}
      }
       \subfigure[\scriptsize{ $d =3$ and $a=1$.}]{
     \includegraphics[width=0.3\textwidth]{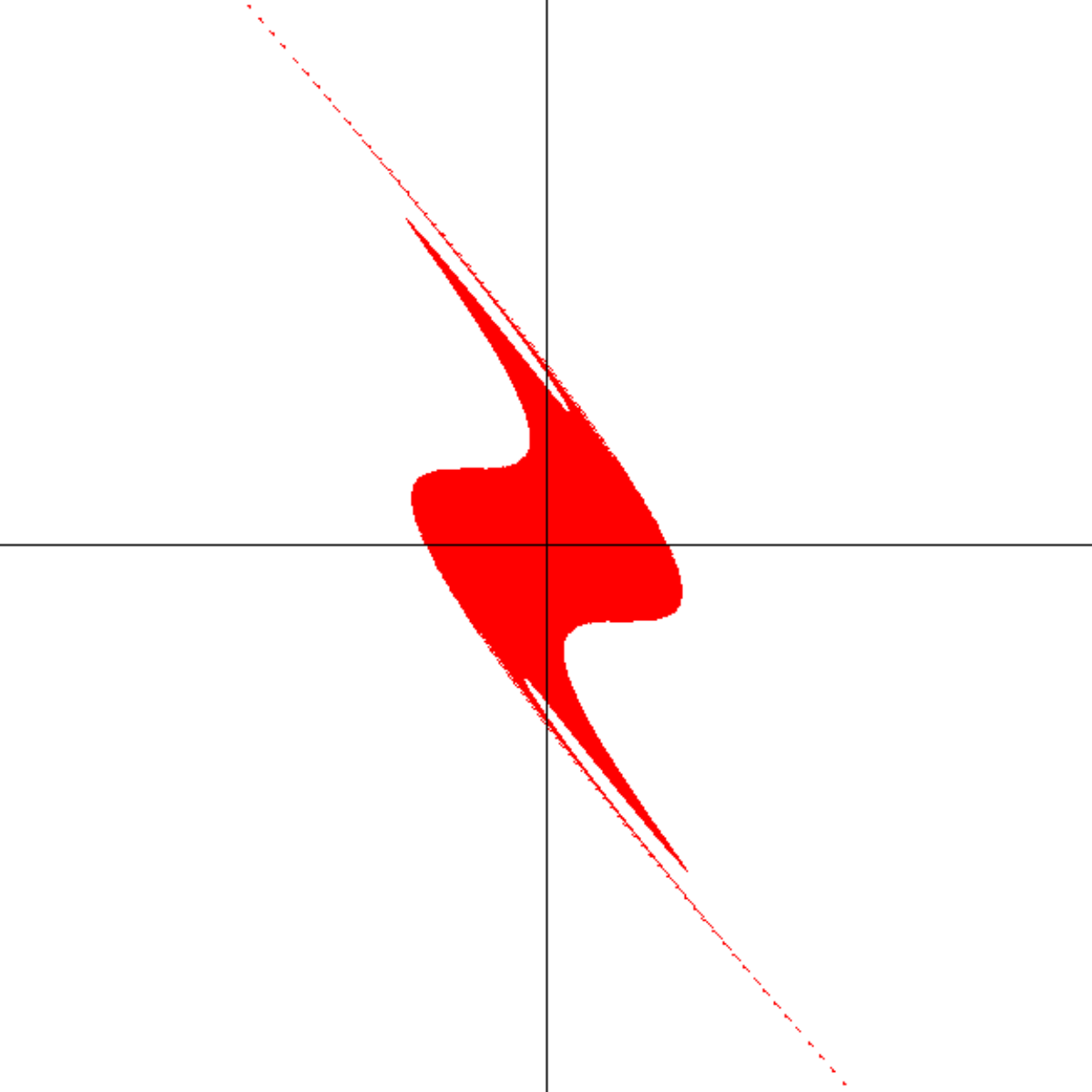}
      }      
    \subfigure[\scriptsize{$d =3$ and $a=-1$.}]{
     \includegraphics[width=0.3\textwidth]{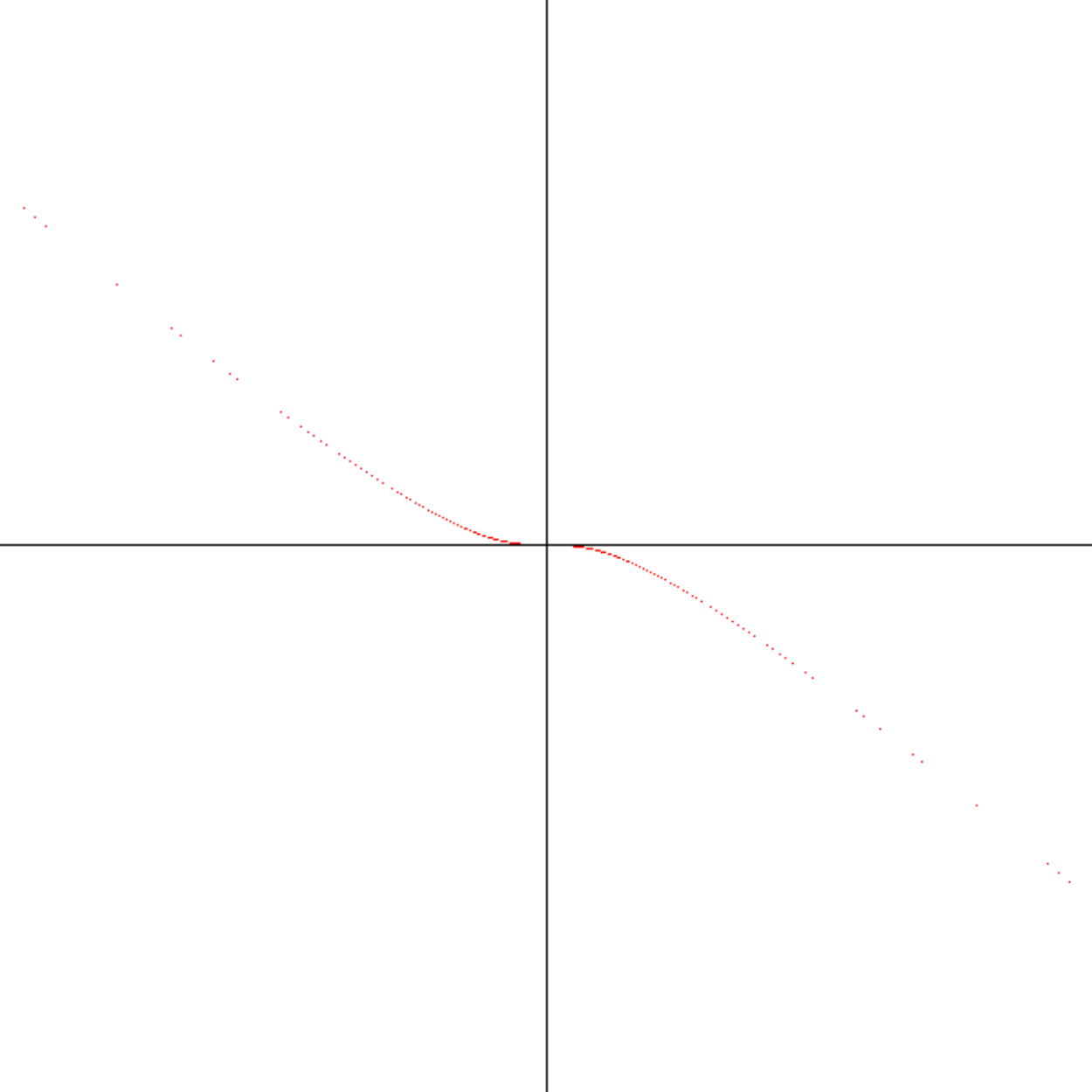}
}
      
 \caption{\small{ Phase plane of the map $T_{a,d}$ for different values of $a$ and $d$. The basin of attraction $\mathcal A_{a,d}$ is shown in red. }}
    \label{fig:dynamical_plane}
    \end{figure}

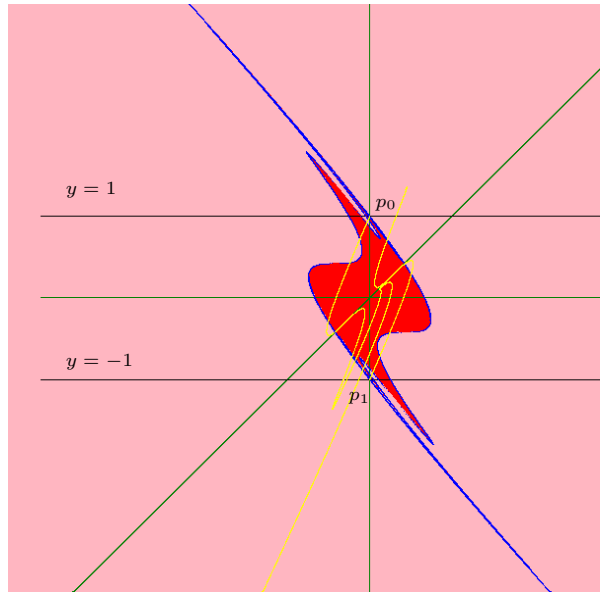
\begin{figure}[ht]
    \centering
     \includegraphics[width=0.5\textwidth]{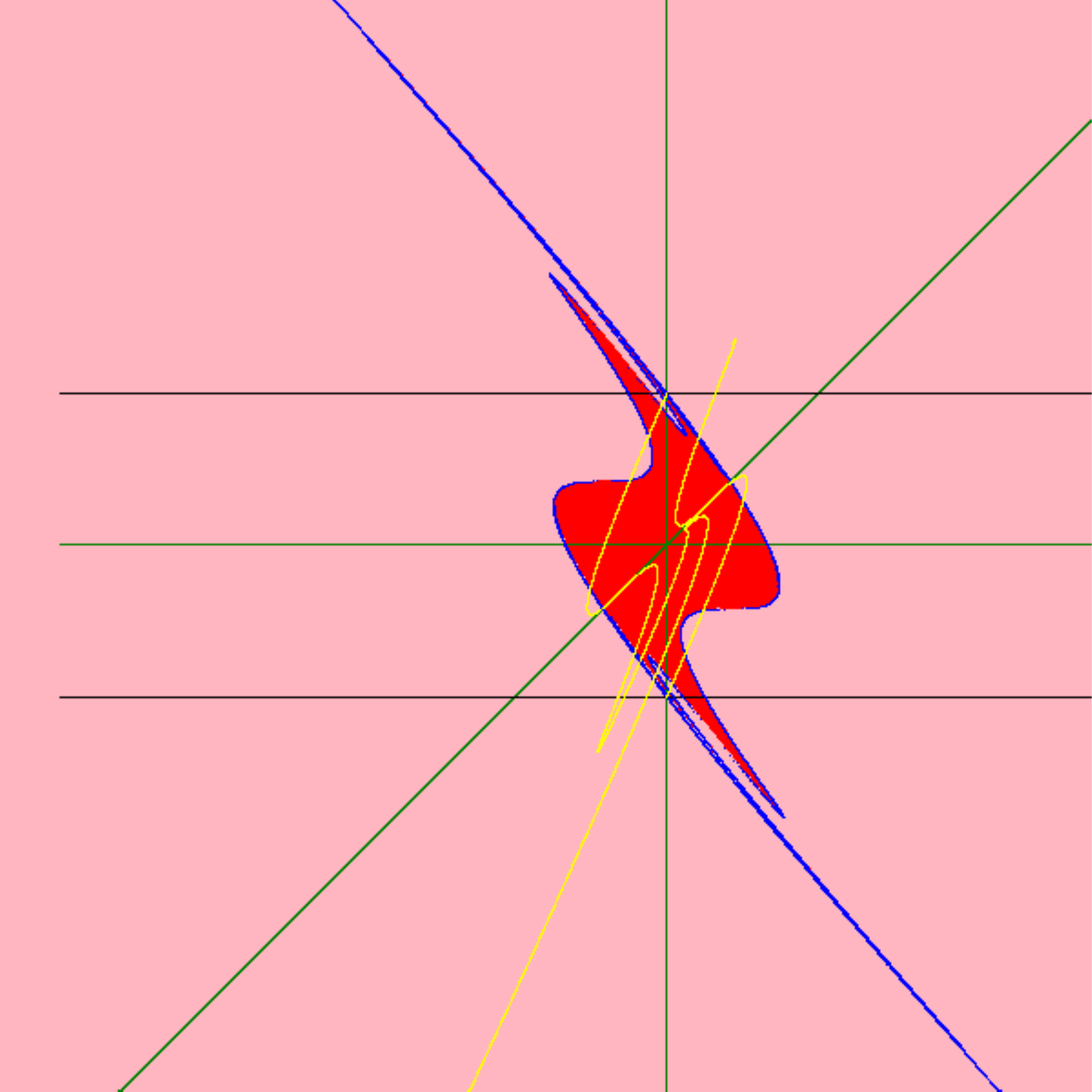}
      \put(-200,150){\tiny $y=1$}
      \put(-200,85){\tiny $y=-1$}
      \put(-84,145){\tiny $p_0$}
      \put(-94,72){\tiny $p_1$}
       \caption{\small{The homoclinic intersection between the stable (blue) and the unstable (yelow) manifolds of the hyperbolic  two-cycle $\{p_0,p_1\}$. The picture also illustrates that the stable manifold of the cycle is related to the boundary of $\mathcal A_d(0)$ (in red) for $d\geq 3$ odd (the picture is done for $d=3$).}}       
    \label{fig:intersection_global}
    \end{figure}

We finish with an important remark, somehow complementary, on the previous result to calibrate their value. On the one hand, from construction, system \eqref{eq:T} encodes the information of system \eqref{eq:S3} as long as $(x,y)\approx (0,0)$. But if one reads carefully Theorem A, we see that it does not refer to the dynamics in a given small neighbourhood of the origin, as for instance Theorem A(b) is showing that $\mathcal A_{a,d},\ a>0,$ is unbounded. Hence, {\it a priory}, there is no reason to argue that Theorem A can be {\it transported} to explain the red regions in Figure \ref{fig:secant}. However, comparing the top right picture in Figure \ref{fig:secant} with Figure \ref{fig:dynamical_plane}(a) ($d$ even), or comparing the bottom right picture in Figure \ref{fig:secant} with Figure \ref{fig:dynamical_plane}(b) ($d$ odd) one can immediately see that \eqref{eq:T} and \eqref{eq:S3} share more than expected. A better explanation for this  connection, somehow global, will require future work.

In a companion paper \cite{FGJ24b} we study in more depth the boundary of $\mathcal A_{a,d}(0)$  when $d$ is odd and $a $ is positive and we show 
there is a (topologically transversally) homoclinic intersection between the stable and the unstable manifolds of the hyperbolic two-cycle $\{p_0,p_1\}$ and there are infinitely many periodic points (somehow chaotic dynamics) in $\partial \mathcal A_{a,d}(0)$. See Figure \ref{fig:intersection_global}.

The paper is organized as follows. In Section \ref{sec:Model} we show that $T_{a,d}$ reduces to three cases: $d$ even with $a=1$, and $d$ odd with $a=\pm 1$. In Section \ref{sec:local} we study the series expansions of the stable and center invariant  manifolds of the origin. Theorems  A(a) and A(b) are proven in Sections \ref{sec:d_even} and \ref{sec:d_odd_a1}, respectively. Finally, in Section \ref{sec:d_odd_aminus1} we prove Theorem A(c).

\section{Preliminaries and local dynamics near the origin}\label{sec:Model}

A preliminary  simple step is to show that, given $d\ge 2$ fixed, for most values of $a\ne 0$ the maps of the family $T_{a,d}$ in \eqref{eq:T} are conjugate to each other, so that we only need to deal with one or two particular values of $a$. See Corollary \ref{coro:ad} below.

\begin{lemma}\label{lem:a_scale}
We have the following statements.
\begin{enumerate}
\item [(a)] If $d$ is even and $a_1$ and $a_2$ are such that $a_1 a_2 \neq 0$ then $T_{a_1,d}$ is conjugate to $T_{a_2,d}$. 
\item[(b)] If $d$ is odd and  $a_1$ and $a_2$ are such that $a_1 a_2 >0$ then $T_{a_1,d}$ is conjugate to $T_{a_2,d}$. 
\end{enumerate}
\end{lemma} 

\begin{proof}
The conjugation will be a rescaling. 
Given any $\mu \in \R$ we have that 
\[
T_{a,d}(\mu x, \mu y) = \mu \left( y - a\mu^{d-1} (x+y)^d , y - 2a\mu^{d-1}  (x+y)^d\right)  = \mu T_{a \mu^{d-1},d} (x,y).
\]
Given $a_1, a_2\ne 0$ we take
$$
\mu := ( a_2/a_1)^{1/(d-1)}. 
$$

If $d$ is even and $a_1$ and $a_2$ are two parameters with $a_1 a_2 \neq 0$  we immediately have
$$
T_{a_1,d}(\mu x, \mu y)= \mu T_{a_2,d}(x,y).
$$ 
If $d$ is odd the same is true but the existence of the $(d-1)$-root requires the condition $a_1a_2>0$.
\end{proof}

\begin{corollary} \label{coro:ad}
To study the dynamics of the family of maps given by \eqref{eq:T} it is enough to consider the cases $\{a=1,\ d\geq 2\}$ and $\{a=-1,\ d\geq 3, \ d \ \mbox{odd}\}$. 
\end{corollary}

 To avoid heavy notation (depending on the parameter $a=\pm 1$) in what follows  we assume $a=1$. We will deal with the case $a=-1$ (for $d$ odd) in Section \ref{sec:local}, Remark \ref{remark:amenys1}, and in Section \ref{sec:d_odd_aminus1}. In particular, when $a=1$, we will use the simplified notation
\begin{equation}\label{eq:Td}
T_{d}(x,y):=T_{1,d}(x,y)=
 \left(
\begin{array}{l}
y - (x+y)^d \\ 
y - 2(x+y)^d
\end{array}
\right).
\end{equation}

\begin{lemma}
	\label{lem:lines}  
We have
\begin{itemize}	
	\item[(a)] If $d$ is even, 
	$T_d$ sends $\R^2$ onto $\{x\ge y\}$. The map $T_d$ has two inverses
	\begin{equation}\label{eq:inverse_d_even}
		T^{-1}_{\pm, d} (x,y) = \left( -2x+y \pm  \left( x-y\right)^{1/d} , \, 
		2x-y \right)
\end{equation}  
	which determine two one to one maps  
	$T^{-1}_{+,d}:\{x\ge y\} \to \{x\ge -y\}$   and $T^{-1}_{-,d}: \{x\ge y\}\to \{x\le -y\}$.

	Moreover,  for any $x_0\in \R$,  $T_d$ maps the line $y=-x+x_0$ onto the line $y=x-x_0^d$ in a one-to-one way. 	
	\item[(b)] If $d$ is odd, the map $T_d:\R^2\to \R^2$ is a homeomorphism onto $\R^2$ and its inverse map is real analytic in $\R^2 \setminus \{x=y\} $, but not differentiable on $\{x=y\}$. Its inverse is  given by
	\begin{equation}\label{eq:T_inverse_odd}
		T^{-1}_{d} (x,y) = \left( -2x+y + \left( x-y\right)^{1/d} , \, 2x-y \right).
	\end{equation}
	Moreover, for any $x_0\in \mathbb R$, the map $T_d$ maps bijectively the line $y=-x+x_0$ onto the line $y=x-x_0^d$.
\end{itemize}
\end{lemma}

\begin{proof}
All statements   come from  direct computations. 
\end{proof}

\section{Local dynamics around the origin: The stable and the center manifolds} \label{sec:local}

The origin is the only fixed point of the map $T_{d}$  in \eqref{eq:Td}. In this section we obtain information on the local dynamics near the origin from the analytic expressions (in series expansion) of the (local) invariant manifolds. The derivative of $T_d$ at $(x,y)$ is given by
\begin{equation}\label{eq:DT}
DT_{d}(x,y)=
 \left(
\begin{array}{cc }
 - d (x+y)^{d-1}&  \ \   1 -  d (x+y)^{d-1}  \\ 
-2d (x+y)^{d-1} & \ \ 1 -2d (x+y)^{d-1} 
\end{array}
\right),
\end{equation}
and therefore  
\[
DT_{d}(0,0)=
 \left(
\begin{array}{cc }
0&   1 \\ 
0 & 1 
\end{array}
\right).
\]
The matrix $DT_{d}(0,0)$ is independent of the parameter $d$. Its eigenvalues are $0$ and $1$ with associated eigenvectors $v_1=(1,0)$ and $v_2=(1,1)$, respectively. In other words the direction $v_1$ is super-attracting while the direction $v_2$ is neutral. It follows from the general theory of  invariant manifolds of fixed points of maps that there is a stable invariant manifold of $(0,0)$ being tangent to $v_1$ and a (non-unique) center invariant manifold of $(0,0)$ being tangent to $v_2$. We denote them $W^s_{d}(0)$ and $W^c_{d}(0)$, respectively. According to the general theory, $W^s_{d}(0)$ is analytic and $W^c_{d}(0)$ is $C^k$ for all $k\ge 1$. 
Even if $W^c_{d}(0)$ may not be unique, all its Taylor coefficients are uniquely determined.  

 More concretely, the local invariant manifolds can be parametrized as graphs
\begin{equation}
\label{def-inv-man-locals}
W^s_{d,\loc}(0) = \{ (x, \varphi^s_{d}(x) ) \mid \,  |x| < \varepsilon_0 \}  \qquad \text{ and } \qquad W^c_{d,\loc}(0) = \{ (x, \varphi^c_{d}(x) ) \mid \,  |x| < \varepsilon_0 \} ,
\end{equation}
for some $\varepsilon_0 >0$,  where 
\begin{equation} \label{eq:stable_center_local}
\varphi^s_{d}(x) = \sum_{n=2}^\infty \alpha_n(d) x^n \qquad  \mbox{and} \qquad   
\varphi^c_{d}(x) = x+ \sum_{n=2}^\infty \beta_n(d) x^n .
\end{equation}
We also denote by $R^s_{d}$ and $R^c_{d}$ the maps which encode the induced dynamics on the invariant manifolds. Thus, locally, we have 
\begin{equation}\label{eq:invariant}
T_{d} \left( x,\varphi_{d}^s(x) \right)  = \left( R^s_{d}(x), \varphi_{d}^s(R^s_{d}(x)) \right) 
\qquad \mbox{and} \qquad 
T_{d}\left(x,\varphi_{d}^c(x) \right)= \left( R^c_{d}(x), \varphi_{d}^c(R^c_{d}(x)) \right), 
\end{equation}
respectively. 

See \cite{HPS77,ParametrizationMethod1,ParametrizationMethod3} for a general discussion on the theory of local invariant manifolds. In the next lemma we provide the structure of the Taylor expansion of $\varphi^s_{d}$ and $\varphi^c_{d}$. The lower order terms will determine the local dynamics near the origin.

\begin{lemma} \label{lem:properties_invariant}
Let $d\geq 2$. The Taylor series of 
$\varphi^s_{d}$ and $\varphi^c_{d}$ have the following structure
\begin{equation} \label{eq:0coeff_stable_center}
\varphi^s_{d}(x)= x^d  \displaystyle \sum_{k=0}^{\infty} \alpha_{d+k(d-1)}(d) x^{k(d-1)} \qquad \mbox{and} \qquad  \varphi^c_{d}(x)= x+ x^d  \displaystyle \sum_{k=0}^{\infty} \beta_{d+k(d-1)}(d) x^{k(d-1)}.
\end{equation}
Moreover, $\alpha_d(d)= 2$, $\alpha_{2d-1}(d)=4d$,  $\beta_d(d)= -2^d$ and $\beta_{2d-1}(d)=-3d2^{2d-1}$ and thus we have that
$$
\varphi_d^s(x)=2x^d+ \mathcal O(x^{2d-1})
\qquad \mbox{and} \qquad \varphi_d^c(x)=x-2^dx^d+\mathcal O(x^{2d-1})
$$
and the one-dimensional dynamics induced by $T_d$ on the stable and center manifolds are governed by
\begin{equation}\label{eq:dynamics}
R^s_d: x \mapsto x^d+ \mathcal O(x^{2d-1}) \qquad \mbox{and} \qquad R^c_d: x \mapsto x-4^4x^d+ \mathcal O(x^{2d-1}),
\end{equation}
respectively.
\end{lemma}
See Figure \ref{fig:local_behavior}(a) and (b) for the induced dynamics of the map $T_d$ on the invariant manifolds.

\begin{proof} To simplify the notation below we introduce the symbol $\{\cdot \}_n$ so that if $\Phi$ is a  formal series around the origin, we write
$$
\Phi(x)=\sum_{n\geq 0} \{\Phi\}_n \ x^n.
$$ 
We prove \eqref{eq:0coeff_stable_center} for the case of the stable manifold $W^s_{d}(0)$ (see  \eqref{eq:stable_center_local}). Using that the stable manifold  is an invariant graph for $T_d$ we obtain that  if $W^s_{d}(0)= \graph \varphi^s_{d}$ then
\begin{equation}\label{eq:invariance}
 \varphi^s_{d}(x) - 2 \left[ x+ \varphi^s_{d}(x) \right]^d = \varphi^s_{d} \left (  \varphi^s_{d}(x) -  \left[ x+ \varphi^s_{d}(x) \right]^d \right).
\end{equation}
From the above equation, some computations show that, on the one hand $\alpha_2(2)=2$ and $\alpha_2(d)=0$ for all $d\geq 3$, and on the other hand, for all $n \geq 3$ we have that  $\alpha_n(d)$ in \eqref{eq:stable_center_local} can be written recursively as
{
\small
\begin{equation}\label{eq:term_by_term}
\alpha_n(d)=  2 \left \{  \left( x+  \displaystyle \sum_{ j=2} ^{n-1} \alpha_j(d) x^j    \right)^d \right \}_n  +   \displaystyle \sum_{ i = 2} ^{n-1}    \alpha_i(d)   \left \{  \left(  \displaystyle \sum_{j=2} ^{n-1} \alpha_j(d) x^j  -  \left( x+  \displaystyle \sum_{ j=2} ^{n-1} \alpha_j(d) x^j    \right)^d  \right)^i   \right \}_n.
\end{equation}
}
Proving \eqref{eq:0coeff_stable_center} for the stable manifold is equivalent to see that in \eqref{eq:stable_center_local} the coefficient $\alpha_n(d)=0$  for all $n\geq 2$ such that $n-d$ is not a multiple of $d-1$, or equivalently, not of the form $n=d+k(d-1)$ for $k\geq 0$. We argue by induction. We claim that for any $N\ge 1$, up to order
$$
n=d+(N-1)(d-1)
$$ 
the stable manifold writes as
\begin{equation} \label{eq:forma_phi}
 x^d  \displaystyle \sum_{ k =0} ^{N-1} \alpha_{d+k(d-1)}(d) x^{k(d-1)} =: x^d \Psi(x^{d-1}).
\end{equation}
When $N=1$ the result is true since $\alpha_n(d) = 0$ for $2 \leq n \leq d-1$ and $\alpha_d(d)=2$. Indeed,  from \eqref{eq:term_by_term}, $a_j(d)=0$ implies $\alpha_{j+1}(d)=0$ for all $j=1,\ldots, d-2$. Also, $\alpha_d(d)=2$ since we have a unique term of degree $d$ with coefficient $2$ associated to the first $\{\Phi\}_n$ term in the right hand side of \eqref{eq:term_by_term}. 

Assuming the claim is true for $N$, we are going to prove that in the right hand side of \eqref{eq:term_by_term} the coefficients
$\alpha_{n+j}(d) x^{n+j} $, $j\ge 1$, are involved in terms of order $n+d$ or higher. This is easy to check for $j=1$. For $j>1$ the coefficients appear in terms of order bigger or equal than $n+d+1$.    
In the right hand side of \eqref{eq:term_by_term} the first term is
$$
2(x+x^d \Psi(x^{d-1}) + \alpha_{n+1}(d) x^{n+1} + \dots)^d 
$$
and the lower term in which $\alpha_{n+1}(d)$ appears is 
$2dx^{d-1}\alpha_{n+1}(d) x^{n+1} = 2d\alpha_{n+1}(d) x^{n+d}$. 
The second term of  the right hand side of \eqref{eq:term_by_term} can be written as
$$
\alpha_{d}(d) \left(x^d \Psi(x^{d-1})+ \alpha_{n+1}(d) x^{n+1}+\dots -
(x+x^d \Psi(x^{d-1})+ \alpha_{n+1}(d) x^{n+1}+\dots) ^d \right)^d
+ \dots$$
and the lower term in which $\alpha_{n+1}(d) $ appears is 
$2d (2x^d)^{d-1} \alpha_{n+1}(d) x^{n+1} =   \mathcal O(x^{n+d(d-1)+1})$.
This finishes the induction.

Once the expression of $\varphi^s_{d}$ given in \eqref{eq:0coeff_stable_center} is proved and the first terms of the expansion have been calculated we only need to justify the expression in \eqref{eq:dynamics}. For this we compute the image of a point on the  stable invariant manifold only using the lowest term of the series expansion 
\begin{align*}
T_d\left(x,\varphi^s_d(x)\right)& =\left(2x^d-\left(x+2x^d\right)^d, \, 2x^d-2\left(x+2x^d\right)^d\right)
\\& =\left(x^d+ \mathcal O(x^{2d-1}),  -4dx^{2d-1} +  \mathcal O(x^{3d-2)})\right).
\end{align*}
Therefore, the one-dimensional dynamics is given by 
\begin{equation*}
x \mapsto x^d+ \mathcal O(x^{2d-1}). 
\end{equation*}
Similar computations provide the result for $\varphi^c_{d}$. 
\end{proof}

\begin{remark}\label{remark:amenys1}
Using the same arguments as the ones in Lemma \ref{lem:properties_invariant} one can get similar results for the case $d$ odd and $a=-1$. The difference is the sign of some leading  coefficients. More precisely if $d$ is odd and $a=-1$ in the definition of $T_{d}$ we have that
$\alpha_d(d)= -2$, $\alpha_{2d-1}(d)=4d$,  $\beta_d(d)= 2^d$ and $\beta_{2d-1}(d)=-3d2^{2d-1}$ and hence we have that
$$
\varphi_d^s(x)=-2x^d+ \mathcal O(x^{2d-1})
\qquad \text{and} \qquad \varphi_d^c(x)=x+2^dx^d+\mathcal O(x^{2d-1})
$$
and the one-dimensional dynamics induced by $T_d$ over the stable and center manifold are governed by
\begin{equation*} 
x \mapsto -x^d+ \mathcal O(x^{2d-1}) \qquad \text{and} \qquad x \mapsto x+4^4x^d+ \mathcal O(x^{2d-1}),
\end{equation*}
respectively. See Figure \ref{fig:local_behavior}(c) for the induced dynamics of the map $T_d$ over the invariant manifolds in this case.
\end{remark}

\begin{figure}[ht]
    \centering
     \includegraphics[width=0.9\textwidth]{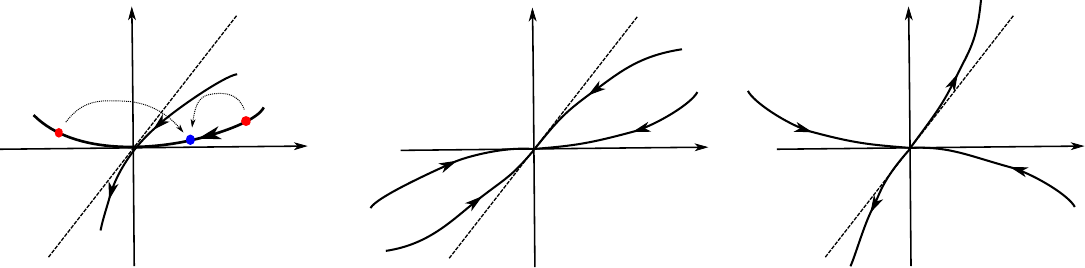}
        \put(-380,-15){\small (a)  $d$ even  and $a=1$}
     \put(-299,55){\small $W^s(0)$}
     \put(-312,72){\small $W^c(0)$}
             \put(-240,-15){\small (b)  $d$ odd  and $a=1$}
          \put(-160,82){\small $W^c(0)$}
     \put(-290,23){\small $W^s(0)$}
                  \put(-110,-15){\small (c)  $d$ odd  and $a=-1$}
     \put(-35,100){\small $W^c(0)$}
     \put(-12,12){\small $W^s(0)$}
           \caption{\small{Local dynamics of  $T_{a,d}$ near the origin.}}
    \label{fig:local_behavior}
    \end{figure}
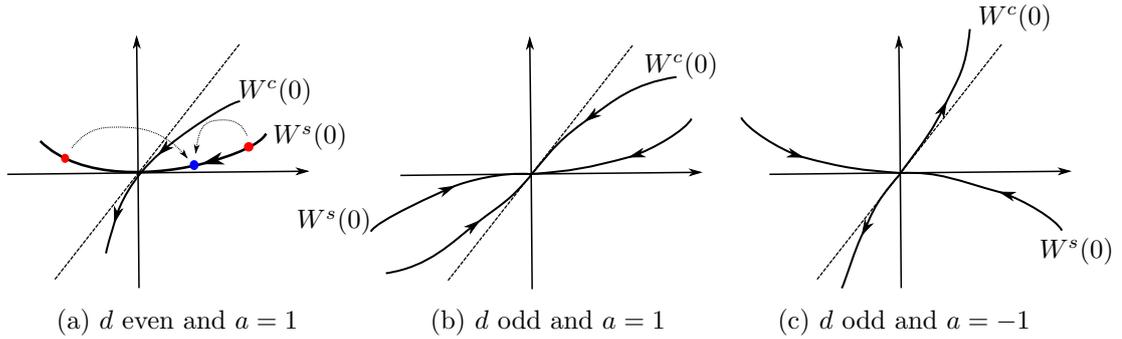
We close this section by completing the discussion above, for the case $d$ even. We have shown in Lemma \ref{lem:properties_invariant} that many coefficients of the series expansion of the stable and center manifolds are zero (no mater the parity of $d$). Next we prove that, for $d$ even, all non-zero coefficients of $\varphi^s_{d}$ are positive. 

\begin{lemma}\label{lem:stable_positive}
Let $d$ be even. Then, $\alpha_\ell(d)\geq 0$ for all $\ell\geq 0$. \end{lemma}
 
 \begin{proof}
In Lemma \ref{lem:properties_invariant} we proved that the coefficients of the series expansion of the analytic expression of the local stable manifold at the origin satisfy certain properties. In particular we proved that all coefficients $\alpha_\ell(d),\ \ell \geq 0,$ of the monomials $x^{\ell}$ with $\ell \ne d+k(d-1)$ for some $k\geq 0$ are zero. Moreover we also proved that $\alpha_d(d)=2$ for all $d\geq 2$. 

We fix $d\geq 2$ even. To simplify the notation we remove the dependence of the coefficients with respect to $d$; that is, we write $\alpha_k:=\alpha_k(d)$. Let $\ga(x)$ be the auxiliary analytic function given by the series expansion
$$
\ga(x) =\sum_{k=d}^\infty \ga_k x^k :=\left( x + \displaystyle \sum_{k = d}^{\infty}\alpha_k x^k  \right)^d.
$$
Note that $\ga_{k+1} $ depends on $\alpha_j$, $d\le j\le k$. The lemma follows from the following claim. 

\vglue 0.2truecm
\noindent {\bf Claim:} \textit{ If $n\ge d$ then, for all $d\le  k\le n$, we have $\alpha_k \geq 2\ga_k\geq 0$.}
\vglue 0.2truecm

We prove the claim by induction. For $n=d$ it is obviously true because  $\alpha_d=2$ and $\ga_d=1$. Assuming the claim is true for $n$, from  \eqref{eq:term_by_term} we can write 
\begin{equation} \label{eq:final}
	\alpha_{n+1}=  2 \left\{  \left( x + \displaystyle \sum_{k = d}^{n} \alpha_k x^k  \right)^d  \right\}_{n+1} +  
	\left\{  \displaystyle \sum_{i =d}^{n} \alpha_i  \left[  \displaystyle \sum_{k = d}^{n} \alpha_k x^k - \left( x + \displaystyle \sum_{k = d}^{n} \alpha_k x^k  \right)^d \right]^i  \right\}_{n+1}  .
\end{equation}
The induction assumption implies
\[
\alpha_k \geq 2 \left\{  \left( x + \displaystyle \sum_{j = d}^{k} \alpha_j x^j  \right)^d  \right\}_k    \ge   \left\{  \left( x + \displaystyle \sum_{j = d}^{k} \alpha_j x^j  \right)^d  \right\}_k  
, \qquad d\le  k \leq n.\] 
This implies that  all coefficients of the  terms of order $n+1$ of the second term of the right hand side of \eqref{eq:final} are non-negative (because $i\ge d\ge 2$). 
Then, we conclude from \eqref{eq:final} that $\alpha_{n+1}\geq 2\ga_{n+1}\ge 0$, and the claim follows.

\end{proof}

\section{Proof of Theorem A(${\rm a}$): The case $d$ even and $a=1$}\label{sec:d_even} 

Let $d\geq 2$ be even. From Corollary \ref{coro:ad} we can take $a=1$ to cover all cases $(a\ne 0$). We simplify the notation writing $T_d:=T_{1,d}$ and  $\mathcal A_d(0):=\mathcal A_{1,d}(0)$. We will show that the origin belongs to the boundary of the basin, that the basin is contained in the upper half plane and that its boundary is the stable manifold of the origin.  

Let us introduce some notation. Given $(x_0,y_0)\in \R^2$ we will write $(x_k,y_k) = T_d^k(x_0,y_0)$ for $k\ge 0$. Set
$$
R_d:= \big(1-\frac{1}{d}\big)  (2d)^{\frac{-1}{d-1}}  .
$$
Note that $R_2 = 1/8$ and, in general, $R_d<1$. Finally let 
\[
\TT= \{(x,y)\mid \ y\leq x\} \qquad \text{and} \qquad \TT_{R_d}= \{(x,y)\in \TT \mid \ y\ge 0,\ 0\le  x \le R_d\}.
\]

Since the proof of Theorem A(a) is quite long we split the arguments into several lemmas. The first one is just an observation.

\begin{lemma}\label{lem_easy}
We have that $(x_k,y_k)\in \TT$ for $k\ge 1$ and then the sequences
$\{x_k\}_{k\ge 1}$ and $\{y_k\}_{k\ge 1}$ are monotonically decreasing.
Moreover, while $x_k > - y_k$ the sequences are strictly decreasing.
\end{lemma}

\begin{proof}
The first assertion follows from Lemma \ref{lem:lines}(a). The second one follows directly from the inequalities:
\begin{align}
	\label{decreixx}
x_{k+1} & = y_k - (x_k+y_k)^d \le y_k \le x_k, \quad k\ge 1 \quad (k\ge 0 \ \ \text{if} \ \ (x_0,y_0)\in \TT), \\
y_{k+1} & = y_k - 2(x_k+y_k)^d \le y_k,   \quad \qquad  k\ge 0.
\label{decreixy}
\end{align}
\end{proof}

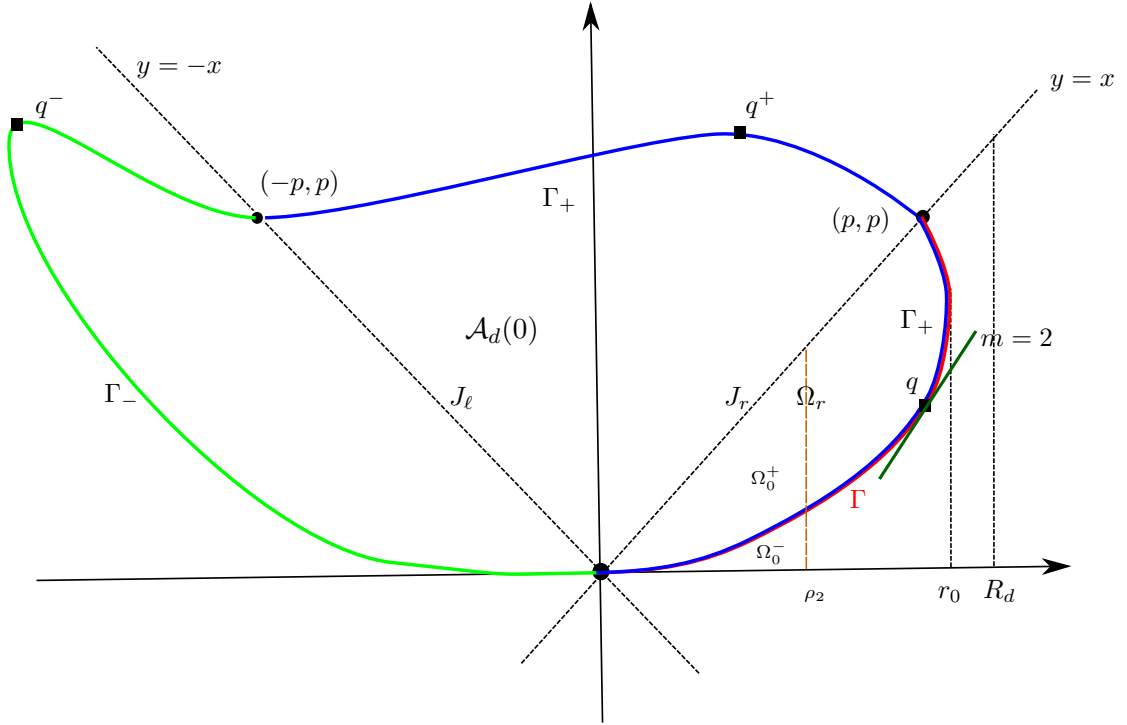
\begin{figure}[ht]
    \centering
     \includegraphics[width=0.9\textwidth]{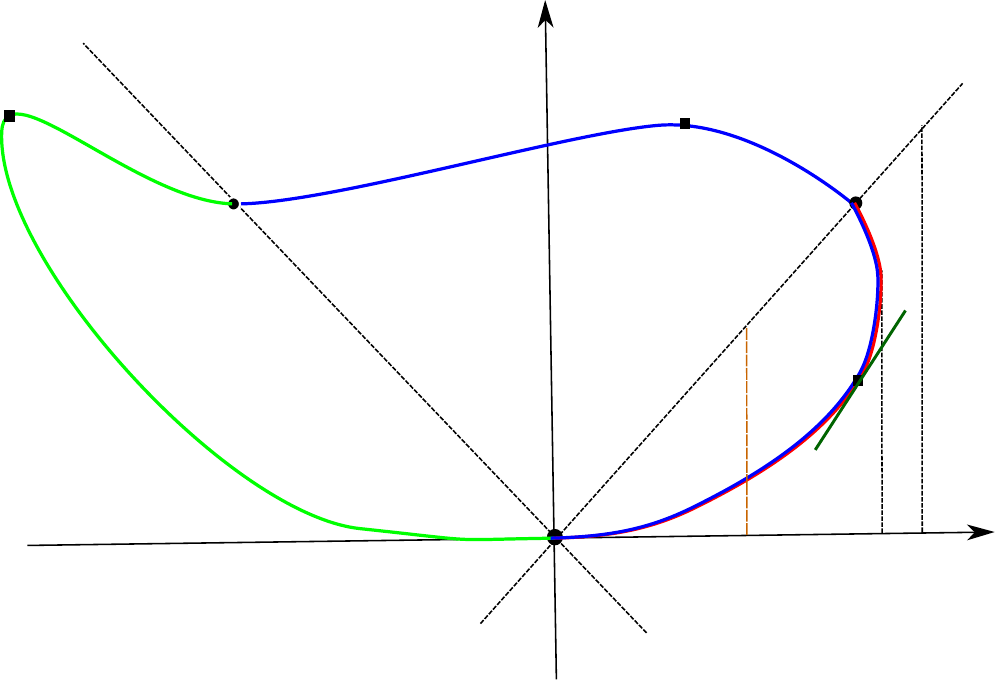}
     \put(-33,47){\small $R_d$}
          \put(-50,47){\small $r_0$}
          \put(-90,187){\small $(p,p)$}
                    \put(-304,201){\small $(-p,p)$}
                    \put(-227,145){$\mathcal A_{d}(0)$}
                    \put(-83,81){\small ${\r \Gamma}$}
                                        \put(-62,126){\small $q$}
                                        \put(-34,142){\small $m=2$}
                                        \put(-198,195){\small $\Gamma_+$}
                                            \put(-64,149){\small $ \Gamma_+$}
                                         \put(-122,230){\small $q^+$}
                                        \put(-388,230){\small $q^{-}$}
                    \put(-362,121){\small $\Gamma_-$}
          \put(-8,240){\small $y=x$}
                    \put(-350,245){\small $y=-x$}
                    \put(-100,47){\tiny $\rho_2$}
                    \put(-120,90){\tiny $\Omega_0^{+}$}
                    \put(-118,62){\tiny $\Omega_0^{-}$}
                    \put(-103,120){\small $\Omega_{r}$}                                      \put(-233,120){\small $J_{\ell}$}
                    \put(-130,120){\small $J_r$}
                    
           \caption{\small{Sketch of the construction of $\mathcal A_{d}(0)$. }}
    \label{fig:basin_d_even}
    \end{figure}

Next lemma shows that $\mathcal A_{d}(0)$ is a bounded set.
  
\begin{lemma}\label{fitaconca}
$\mathcal A_{d}(0)$ is a bounded  set. More concretely, 
$$
\mathcal A_{d}(0) \subset \K := (-5R_d, R_d)\times (0,2R_d) \cup \{(0,0)\}. 
$$
  
\end{lemma}

\begin{proof}
We decompose 
$$\R^2 \setminus \K = \K_1 \cup \K_2 \cup \K_3 \cup \K_4, 
$$    
where
\begin{align*}
\K_1 & = \{(x,y)\mid\ y\le 0\} \setminus \{(0,0)\}, 	\\
\K_2 & = \{(x,y)\mid\ x\ge R_d\}, 	\\
\K_3 & = \{(x,y)\mid\ y\ge 2R_d\} , 	\\
\K_4 & = \{(x,y)\mid\ x \le -5R_d, \ 0< y <2R_d\} 
\end{align*}
and we argue that $\A_d(0) \cap \K_j = \emptyset$, $1\le j\le 4 $, so that 
$\A_d(0) \subset \K$. 
We will use the following property: by the invariance of $\A_d(0)$ by 
$T_d$ and its inverses we have that $T_{\pm,d}^{-1}$,  $(x_0,y_0) \in \A_d(0)$ if and only if 
 $(x_k,y_k) \in \A_d(0)$ for some $k\ge 0$.
 
Let $(x_0,y_0) \in \K_1$. Since $(x_0,y_0) \neq (0,0)$, if $y_0=0$ then
$y_1=-2x_0^d<0$ and if $y_0<0$ then
$y_1\le y_0<0$. Hence, in both cases 
$y_1<0$ and by \eqref{decreixy} the sequence of iterates cannot converge to $(0,0)$.

Next, we claim that if $(x_0,y_0) \in \K_2$ then $(x_1,y_1 ) \in \K_1$.
Indeed, we consider the line $\{x=x_0\}$ with $x_0\ge R_d$ and we look at the second component of its image $\Psi_1(y):= \pi_y T_d(x_0, y)= y - 2(x_0+y)^d$. Since $d\ge 2 $ is even, $\lim_{t\to \pm \infty} \Psi (y) = -\infty$  and therefore $\Psi_1$ has a global maximum. Actually, it has a unique maximum whose location is obtained from the condition $\Psi_1'(y) =0$ and is 
$y^{(m)}:= 1/(2d)^{1/(d-1)} -x_0 $.
  Then  $\Psi^{(m)}_{1} := \Psi_1(y^{(m)}) = 1/(2d) ^{1/(d-1)}
- 2/(2d)^{d/(d-1)} -x_0 =R_d -x_0$  and therefore $y_1\le 0$.
Moreover, $(x_1,y_1) \ne (0,0) $ because the only preimage of $(0,0 )$ is $(0,0) \notin \K_2$.

Next we take $(x_0,y_0) \in \K_3$ and we claim that $(x_1,y_1 ) \in \K_1 \cup \K_2$. 
Indeed, consider the line 
$\{ y=y_0\}$ with $y_0 \ge 2R$. Its image is contained in the line $\{(u,v) \mid \ v=2u-y_0\}$ and it is   contained in $\K_1 \cup \K_2$ because 
we have that either $u \ge R_d $ and the claim is true or $u < R_d $ and then $v = 2u -y_0 < 2R_d -y_0 \le 0$ and $(x_1, y_1)\ne (0,0)$.  

Finally, if $(x_0,y_0) \in \K_4$ we claim that  
$(x_1,y_1 ) \in \K_1 \cup \K_2$.
Indeed, notice that $x_0+y_0 < -5R_d + 2R_d =-3R_d$ and then $(x_0+y_0)^d > (3R_d)^d$. 
If  $x_1 \ge  R_d $ the claim is true. If  $x_1 < R_d $ then 
$y_1 = y_0 - 2(x_0+y_0)^d = x_1 -(x_0+y_0)^d 
< R_d -(3R_d)^d = (1-\frac{3^d}{2d}(1-\frac1d)^{d-1}) R_d < 0$. 
\end{proof}
\begin{lemma} \label{iteratsenT}
We have that $(x_0,y_0) \in \mathcal  A_{d}(0)$ if and only if $(x_k,y_k)\in \TT_{R_d}$ for all $k \geq 1$. 
\end{lemma}
\begin{proof}
Assume  $(x_0,y_0)  \in \A_d(0)$. By Lemma \ref{lem:lines}(a), $x_k\ge y_k$ for all $k\ge 1$ and by Lemma \ref{fitaconca}, $x_k < R_d$ for all $k\ge 0$. Since the sequences $\{x_k\}$ and $\{y_k\} $ are decreasing for $k\ge 1$, if there exists $ m >0$ such that $y_m <0$, then $y_k\le y_m<0$ for all $k\ge m$ and $(x_k,y_k)$ cannot converge to $(0,0)$. Then,  $y_k\ge 0$ for all $k$ and the limit 
$y^\star = \lim_{k\to \infty} y_k $ exists, $y^\star \ge 0$ and then $x_k \ge y_k \ge 0$. As a consequence   $(x_k,y_k)\in \TT_{R_d}$ for all $k \geq 1$.

Conversely, let $(x_0,y_0)  \in \R^2$  and assume that $(x_k,y_k)\in \TT_{R_d}$ for all $k \geq 1$. Since the sequence $\{y_k\}_{k\geq 0}$ is strictly decreasing and bounded from below by 0 
there exists the limit 
$y^\star = \lim_{k\to \infty} y_k \ge 0$.
From the recurrence 
\begin{equation*}\label{eq:recurrence}
	y_{k+1} = y_k - 2 (x_k + y_k)^d 
\end{equation*} 
we obtain that  $\lim _{k\to \infty} (x_k+y_k) $ exists and it is 0. 
This implies that $-y^\star = \lim _{k\to \infty} x_k \ge 0 $. Then $y^\star =0$ and   $(x_0,y_0)  \in \A_d(0)$.
\end{proof} 

We now turn the attention to $\partial \mathcal A_d(0)$. Our goal is to prove that $\partial \mathcal A_d(0)$ coincides with the global stable manifold of the origin, $W^s_{d}(0)$.
\begin{lemma} \label{lem:cut}
$W^s_d(0)$ cuts the line $\{x=y\}$ at some point $(p,p)$ with $0<p < R_d$.	
\end{lemma}

\begin{proof}
In Lemma \ref{lem:stable_positive} it is proven that the local expression of $W^s_{d}(0)$ is given by the graph of an analytic function $\varphi_d^s(x)= 2 x^d + \dots $ whose series expansion in the $x$-variable has all its coefficients non-negative, and therefore there exists $\rho_1 >0$ such that $\gamma = \{(x,\varphi_d^s(x))\mid \ x\in (0,\rho_1)\}$ is contained in $\TT_{R_d}$. We claim that the extension of this local piece $\gamma$ of $W^s_{d}(0)$ eventually leaves $\TT_{R_d}$. Indeed, assume the contrary.
We globalize $\gamma$  iterating with $T^{-1}_{+,d}$ (see \eqref{eq:inverse_d_even}). 
Let $(x_0,y_0) \in \gamma$ and denote 
 $$
\left(x_{-k},y_{-k}\right)=T^{-k}_{+,d}(x_0,y_0),\qquad  k\geq 0.
$$
We have 
\begin{equation}\label{recpery-}
y_{-k-1} = y_{-k} + 2 (x_{-k}-y_{-k}) > y_{-k}.
\end{equation}
If all $\left(x_{-k},y_{-k}\right) \in \TT_{R_d}$ we have that the sequence $\{y_{-k}\}_{k \geq 0}$ is strictly increasing and bounded, and we conclude that there exists $y^{\star}>0$ such that 
\begin{equation} \label{limityk}
	\displaystyle y^\star=\lim_{k \to \infty} y_{-k} >0.
\end{equation}
Moreover, from \eqref{recpery-} we have that 
$x_{-k} =  (y_{-k} +y_{-k-1} )/ 2  \to y^\star$. 
Now, using the recurrence 
$$
x_{-k-1} = -2x_{-k} + y_{-k} + (x_{-k} - y_{-k})^{1/d},  
$$
we get that $y^\star=0$, which provides a contradiction with \eqref{limityk}.
Finally, since we have seen that $\A_d(0)$ does not meet $\{y=0\}\setminus \{(0,0)\}$ nor $\{x=R_d\}$ the globalization of $\gamma $ has to cross $\{x=y\}$. 
\end{proof}

We denote by $\Gamma$ the piece of the stable manifold $W^s_{d}(0)$ from $(0,0) $ to $(p,p)$ contained in  $\TT_{R_d}$.
We plot $\Gamma$ in red colour in Figure~\ref{fig:basin_d_even}.  
Let $\varphi^{s}_{d}$ be given in \eqref{def-inv-man-locals}.

\begin{lemma} \label{lem:q_x}
The following properties for $\varphi_d^s$ hold.
\begin{itemize}
\item[(a)] There exists a unique point $\bar q=\left(\bar q_x,\bar q_y\right)\in \Gamma$ whose tangent vector has slope $m=1/2$. 
\item[(b)] If we denote by $r_0>0$ the radius of convergence of $\varphi_d^s$ (as a function of a complex variable) then $0<r_0<R_d$, $\varphi^{s}_{d}$ is 
increasing and convex in the interval $(0,r_0) $
and decreasing and convex  $(-r_0-2\varphi^{s}_{d}(r_0),0 )$.
\end{itemize} 
\end{lemma}

\begin{proof}
We observe that since all coefficients of the series expansion of $\varphi_d^s$ are non-negative (see Lemma \ref{lem:stable_positive}) we conclude from Vivanti-Pringsheim's Theorem \cite{Evg} that $\varphi^s_{d}$ as a function of a complex variable has a singularity at $x=r_0>0$ and 
$$
\varphi_0:  = \varphi^s_{d}(r_0)= \displaystyle \sum_{k \geq d}  \alpha_k r_0^k.
$$  
In fact we have that $r_0<R_d< \infty$ and $\varphi_0< 2R_d$ since   $\graph \varphi_d^s\subset W_d^s(0) \subset \mathcal A_d(0)$ and by Lemma \ref{fitaconca}, $\mathcal A_d(0) \subset \K$. In particular $\varphi^s_{d}|_{(0,r_0)}$ is an increasing and convex function
and 
$$
\displaystyle \lim_{x \to r_0^-}  \left(\varphi^{s}_{d}\right)^{\prime} (x) = + \infty.
$$ 
Indeed, if $\left(\varphi^{s}_{d}\right)^{\prime} (r_0) < \infty$, then  
$\left(\varphi^{s}_{d}\right)^{\prime} $ could be extended in a differentiable way for $x>r_0$. The graph close to $x=r_0$ will be the image by $T^{-1}_{+,d}$ of a piece of the graph of $\varphi^{s}_{d}$, say $\gamma_2$, closer to the origin. The piece $\gamma_2$ does not contain the point $(p,p)$ since its image is the point $(-p,p)$ outside $\TT_{R_d}$. Therefore $T^{-1}_{+,d}$ is analytic on $\gamma_2$ and then $\varphi^{s}_{d}$ would be analytic in a neighborhood of $r_0$ which provides a contradiction. This proves statement (a) and provides the existence of $r_0$.

We have the symmetry ($d$ is even)  
\begin{equation} \label{simetria}
T_d\left(x,y\right)=T_d\left(-2y-x,y\right).	
\end{equation}
Hence if $(x,y)\in W_d^s(0)$ then also $(-2y-x,y)\in W_d^s(0)$. 
More concretely, since $(x,\varphi^{s}_{d}(x)) \in W^s_{d}(0)$, $(-x-2\varphi^{s}_{d}(x), \varphi^{s}_{d}(x)) \in W^s_{d}(0)$ and then 
\begin{equation} \label{simetria-phi}
\varphi^{s}_{d} (-x-2\varphi^{s}_{d}(x)) = \varphi^{s}_{d}(x).	
\end{equation}
This means that $\varphi^{s}_{d}$ is defined for $x\in (-r_0-2\varphi^{s}_{d}(r_0),0 )$. Moreover, taking derivatives in \eqref{simetria-phi} we get 
\begin{align}
& (\varphi^{s}_{d}) '(-x-2\varphi^{s}_{d}(x))(-1-2(\varphi^{s}_{d})'(x))  = (\varphi^{s}_{d})'(x), \label{cond-der1}\\	 
& (\varphi^{s}_{d}) ''(-x-2\varphi^{s}_{d}(x))(-1-2(\varphi^{s}_{d})'(x))^2 + (\varphi^{s}_{d}) '(-x-2\varphi^{s}_{d}(x))(-2(\varphi^{s}_{d})''(x))  = (\varphi^{s}_{d})''(x),\label{cond-der2}
\end{align}
and hence we can conclude that $\varphi^{s}_{d}$ is decreasing and convex in $(-r_0-2\varphi^{s}_{d}(r_0),0 )$. Indeed, substituting $(\varphi^{s}_{d}) '(-x-2\varphi^{s}_{d}(x))$ from \eqref{cond-der1} into \eqref{cond-der2} we obtain 
$$
(\varphi^{s}_{d}) ''(-x-2\varphi^{s}_{d}(x))(-1-2(\varphi^{s}_{d})'(x))^2
=   (\varphi^{s}_{d})''(x) \left(1- \frac{2(\varphi^{s}_{d}) '(x)}{1+ 2(\varphi^{s}_{d}) '(x)} \right) >0.
$$
From the previous properties there exists a unique point $\bar q=\left(\bar q_x,\bar q_y\right)\in \Gamma$ whose tangent vector has slope $m=1/2$.
\end{proof}

Let 
\begin{equation*}
\begin{split}
\Omega_{0}^+ &= \{(x,y)\in \TT_{R_d} \mid \, y \ge \varphi_d^s(x),\, 0\le x\le \rho_2 \}, \\
\Omega_{0}^- &= \{(x,y)\in \TT_{R_d} \mid \, 0 < y< \varphi_d^s(x),\, 0\le x\le \rho_2 \},
\end{split}
\end{equation*}
with $\rho_2 < \min \left\{ \frac{1}{2}\frac{1}{(4d)^{1/(d-1)}}, \, \bar q_x\right\}$. See Figure 
\ref{fig:basin_d_even}.

\begin{lemma}\label{lem:Omega}
The domain  $\Omega_{0}^+$ is invariant by $T_d$ and 
		$\Omega_{0}^+ \subset \A_d(0)$. Moreover,   
		$\Omega_{0}^- \cap\A_d(0) =\emptyset$.
\end{lemma}

\begin{proof}
Let $(x,y):=(x_0,y_0) \in \Omega_{0}^+$. The sequence $\{x_k\}_{k\ge 0}$
is decreasing by Lemma  \ref{lem_easy} . So, it is enough to show that $y_1-\varphi^{s}_{d}(x_1) \geq 0$. Indeed 
\begin{equation} \label{fitadiffyph}
\begin{aligned}
y_1-\varphi^{s}_{d}(x_1) &= y - 2(x+y)^d - \varphi^{s}_{d}(y - (x+y)^d) \\
& = y - \varphi^{s}_{d}(x) + \varphi^{s}_{d}(x) - \varphi^{s}_{d}(y - (x+y)^d) - 2(x+y)^d \\
& = y - \varphi^{s}_{d}(x) + H(x,y),
\end{aligned}
\end{equation}
where
\begin{equation} \label{def:H}
H(x,y) = \varphi^{s}_{d}(x) - \varphi^{s}_{d}(y - (x+y)^d) - 2(x+y)^d .
\end{equation}
Taking into account that $\varphi^{s}_{d}$ satisfies the invariance equation
\begin{equation} \label{inveq}
\varphi^{s}_{d}(x) - 2(x+\varphi^{s}_{d}(x))^d  = \varphi^{s}_{d}(\varphi^{s}_{d}(x) - (x+\varphi^{s}_{d}(x))^d), 	
\end{equation}
$H$ can be rewritten as 
\begin{equation} \label{formulaperhtilde}
	\begin{aligned}
		H(x,y) & = \varphi^{s}_{d}(\varphi^{s}_{d}(x) - (x+\varphi^{s}_{d}(x))^d) - \varphi^{s}_{d}(y - (x+y)^d)  + 
		2(x+\varphi^{s}_{d}(x))^d   - 2(x+y)^d \\
		& = \int_0^1\Big[ \frac{d}{dt}\varphi^{s}_{d}(\xi_t-(x+\xi_t)^d)(1-d(x+\xi_t)^{d-1})
		+2d(x+\xi_t)^{d-1} \Big](\varphi_d^s(x) -y) \,dt \\
		& =: (\varphi^{s}_{d}(x) -y) \wh H(x,y),
	\end{aligned}
\end{equation}
where 
$\xi_t= y +t (\varphi^{s}_{d}(x) -y) $ and hence since $(x,y)\in \Omega_0^+$ we have 
$0 < \varphi_d^s(x) \le \xi_t \le y \le x\le \rho_2$.

From \eqref{fitadiffyph} and \eqref{formulaperhtilde} we have 
$$
y_1-\varphi^{s}_{d}(x_1)=\left(y-\varphi^{s}_{d}(x)\right)\left(1-\wh H(x,y)\right),
$$ 
so that it is enough to see that $\wh H(x,y)<1$.

Given $x\in (0, \rho_2)$, we introduce  $\Psi_2(\xi)=\xi-(x+\xi)^d $ 
for $\xi \in (0 , x)$. The function $\Psi_2(x)$  is concave and we have that $\Psi_2(0)=-x^d <0$, with $-x^d >-\rho_2^d  $ and 
$\Psi_2(x)= x-(2x)^d = x(1-2^d x^{d-1}) > x(1-1/(2d) ) >0$ by one of the conditions in the definition of $\rho_2$.
Hence, $\xi_t-(x+\xi_t)^d  \ge -\rho_2^d$ for all $t\in [0,1]$.
Note that, by the fact that the coefficients of the expansion of
$\varphi^{s}_{d}$ are non-negative (Lemma \ref{lem:stable_positive}), at the symmetric point the absolute value of the derivative is smaller, i.e. for $x\in[0,r_0) $, 
$|(\varphi^{s}_{d})'(-x ) | \le (\varphi^{s}_{d})'(x )$  so that, for $-x^d < \zeta <0$, $|(\varphi^{s}_{d})'(\zeta ) | \le (\varphi^{s}_{d})'(-\zeta ) \le (\varphi^{s}_{d})'(x^d )
\le (\varphi^{s}_{d})'(\rho_2 ) \le 1/2$.
Then 
\begin{equation}
\label{condicio-H-1}
	|(\varphi^{s}_{d})'(\xi_t-(x+\xi_t)^d)| <1/2.
\end{equation}
Moreover,   
\begin{equation}
	\label{condicio-H-2}
2d (x+\xi_t)^{d-1} < 2d(2x)^{d-1} \le 1/2. 
\end{equation}
By \eqref{condicio-H-1} and \eqref{condicio-H-2} we obtain 
$\wh H(x,y) < 1$
and therefore
we obtain that the iterates stay in the same side of $\graph \varphi^{s}_{d}$.

Now we deal with $\Omega^-_0$. To prove that 
	$\Omega_{0}^- \cap\A_d(0) =\emptyset$ we will see that if $(x_0,y_0) 
\in \Omega_{0}^- $ then not all its iterates can remain in  
$\Omega_{0}^- $. Assume the contrary. To simplify the estimates we do a ($x$-depending) translation to put the (local) stable manifold at $\{y=0\}$. Actually, we make the  change
$C(x,y) = (x,y+\ph^s_d(x))$. The transformed map is 
$$
\wh T_{d}\left(
\begin{array}{l}
	x \\ 
	y
\end{array}
\right)
= 
\left(
\begin{array}{l}
F(x,y) \\ 
G(x,y)
\end{array}
\right) :=
\left(
\begin{array}{l}
	y +\ph^s_d (x) - (x+y+\ph^s_d (x))^d \\ 
y +\ph^s_d (x) - 2(x+y+\ph^s_d (x))^d - 
\ph^s_d\big( 	F(x,y) \big)
\end{array}
\right).
$$
  
The domain $\Omega^-_0$ is transformed into 
$$
\wh \Omega^-_0 = \{ (x,y)\mid \   0<x<\rho_2, \,  -\ph^s_d (x) < y < 0 \}.
$$
Let $(x_0,y_0) \in \wh \Omega^-_0 $. 
We use again the notation
$(x_k,y_k) = \wh T_d^k(x_0,y_0)$ for $k\ge 0$. 

Let  $\rho_3 \in (0,\rho_2] $ be such that 
$$0< \ph^s_d(x) < 3x^d \qquad \text{ for } \quad x\in (0,\rho_3).
$$

Assume that $(x_k,y_k)\in \wh \Omega^-_0 $ for all $k\ge 0$. 
Since $d$ is even, we also have $0< x_{k+1} \le 	y_k +\ph^s_d (x_k)
< \ph^s_d (x_k) \le x_k$. Then $\{x_k\}$ is also decreasing and 
$$
x_{k} =y_{k-1}+\varphi_d^s(x_{k-1})-\left(x_{k-1}+y_{k-1}+\varphi_d^s(x_{k-1})\right)^d \le \ph^s_d (x_{k-1}) -x^d_{k-1} < 2 x^d_{k-1},
$$
and inductively we get 
\begin{equation} \label{fitaxk}  
x_k < 2^{\left(\frac{d^{k}-1}{d-1}\right)} x_0^{d^k} < (2^{1/(d-1)} x_0)^{d^k}.
\end{equation}
Note that since $2^{1/(d-1)} x_0 < 2^{1/(d-1)} \rho_2 <1/2$ then $x_k\to 0$. We have  
$$
G(x,0) = \ph^s_d(x) - 2(x+\ph^s_d(x))^d  - \ph^s_d(\ph^s_d(x) - (x+\ph^s_d(x))^d)=0
$$
by the invariance  equation \eqref{inveq}, 
and
$$
G_1(x) :=
\frac{\partial G}{\partial y}(x,0) = 1 - 2d(x+\ph^s_d(x))^{d-1}-
(\ph^s_d)'(F(x,0)) (1- d(x+\ph^s_d(x))^{d-1} )= 1 - 2d x^{d-1} +\dots
$$
so that 
$$
G(x,y)= G_1(x)y +G_2(x,y) \qquad \text {with} \quad G_2(x,y) =  \mathcal O(y^2).
$$
There exists $\rho_4 \in (0,\rho_3] $ such that 
$$
G_1(x) > 1-\nu x^{d-1} \qquad \text{and} \qquad |G_2(x,y)| <M |y|^2, \quad x\in (0,\rho_4), \quad (x,y)\in \wh \Omega^-_0 , 
$$
for some $\nu > 2d$ and $M>0$.
Then, taking an iterate $(x_k, y_k)$ such that  $x_k <\rho_4$ and relabeling it by  $(x_0, y_0)$ and, starting again the iteration, we have
\begin{align}
y_{k+1} = G(x_k, y_k) \le  & (1-\nu x_k^{d-1}) y_k + My_k^2 \nonumber \\
 < & (1-\nu x_k^{d-1}-M\ph^s_d(x_k))y_k \le  (1-b x_k^{d-1}) y_k	,
\label{iteracioyhat}
\end{align} 
where $b= \nu + 3\rho_4 M $. 
Iterating \eqref{iteracioyhat} we obtain
\begin{align*}
y_k & < \prod_{j=0}^{k-1} (1-bx_j^{d-1}) y_0 
= y_0 \exp \sum_{j=0}^{k-1} \log (1-bx_j^{d-1}).
\end{align*}
The series 
$\sum \log (1-bx_j^{d-1}) $ is convergent since $bx_j^{d-1}$ tends to zero and $\log (1+x) > (2\log 2 ) x$ if $x\in (-1/2, 0)$. Then, 
$y_k <y_0 \exp ( S_0 )$ where $S_0= \sum_{j=0}^{\infty} \log (1-bx_j^{d-1}) $.
This means that $y_k$ is less than some negative number so that  $y_k$ cannot converge to 0 and therefore $(x_0, y_0) \notin \A_d(0)$.

If $(x_0, y_0)\in   \Omega_0^-$, assume that all its iterates stay in $\TT_{R_d}$. Then the  sequences $\{x_k\}$ and $\{y_k\}$ are decreasing and there exists $m\ge 0$ such that $x_m<\rho_4$ and, by the previous estimates,  $(x_m, y_m)\notin \A_d(0)$.  
\end{proof}

\begin{proof}[Proof of Theorem A(a)]
Let $\Omega_r$ be the closure of the bounded domain whose boundary is the simple closed curve formed by the concatenation of $\Gamma$ and 
$J_r:=\{(x,y)\mid\  x=y, \, 0< x < p \}$ (the meaning of $r$ is {\it right}, in contrast of the later notation $J_\ell$ for {\it left}).  See Figure \ref{fig:basin_d_even}.  The domain $ \Omega_r$ is invariant by $T_d$ since the iterates cannot jump across the boundary. Moreover, there exists $m\ge 1$ such that 
$T^m_d(\Omega_r) \subset \Omega ^+_0$. 
Then $\Omega_r \subset \A_d(0)$.  By Lemma  \ref{iteratsenT}, to obtain $\A_d(0)$ we only need to take one preimage of $\Omega_r$ by $T_d$.

In the light of Lemma \ref{lem:lines}(a) we write  
$$
\Gamma_{\pm} := T^{-1}_{\pm,d}\left(\Gamma\right), \qquad 
\Omega_{\pm} := T^{-1}_{\pm,d}\left(\Omega_r \right).
$$
Clearly, the sets  $\Omega_{\pm} $ are contained in $\A_d(0)$. The boundaries of   $\Omega_{\pm} $ are the images of the boundaries of $\Omega_r$ by $T_{\pm,d}^{-1}$. Consequently, we have
\begin{equation*}
\partial \Omega_+=\Gamma_+ \cup J_{\ell} \qquad  \text{and} \qquad
\partial \Omega_-=\Gamma_- \cup J_{\ell},
\end{equation*}
where $\Gamma_+:=T_{+,d}^{-1}(\Gamma)$ is a curve contained in 
 $\{y\ge -x\}$ which joints $(0,0)$ with $(-p,p)$, $\Gamma_-:=T_{-,d}^{-1}(\Gamma)$ is a curve contained in $\{y\le -x\}$ which joints $(0,0)$ with $(-p,p)$, and 
$J_\ell:=T_{+,d}^{-1}(J_r) = \{(x,y)\mid\  y=-x, \, -p< x < 0 \}$.   Notice that every point in $\Gamma \setminus \{(0,0)\cup (p,p)\}$ has two preimages while 
$$
T^{-1}_{+,d}\left(0,0\right)=T^{-1}_{-,d}\left(0,0\right)=(0,0) \qquad \mbox{and} \qquad T^{-1}_{+,d}\left(p,p\right)=T^{-1}_{-,d}\left(p,p\right)=\left(-p,p\right).
$$
See Figure \ref{fig:basin_d_even}. Accordingly, the curves $\Gamma_{\pm}$ joint the points $\left(0,0\right)$ and $\left(-p,p\right)$, they are mapped bijectively onto $\Gamma$ by $T_d$ and determine the boundary of the basin of attraction of the origin. That is,
$$
\mathcal A_d(0) = \Omega_+\cup \Omega_- \cup \Gamma_+ \cup \Gamma_-.
$$
In Figure  \ref{fig:basin_d_even} we draw $\Gamma_-$ in green and $\Gamma_+$ in blue. This finishes the proof of Theorem A(a). 
\end{proof}

We can add some extra information about the geometry of $W_d^s(0)$. See Figure \ref{fig:basin_d_even}.
On the one hand, direct computations from \eqref{eq:inverse_d_even} imply that if $u=\left(u_1,u_2\right)$ is the tangent vector of $W_d^s(0)$ at the point $(p,p)$ then 
$$
DT^{-1}_{\pm,d}(p,p)(u)
= 
\left(
\begin{array}{c}
	\infty \\ 2u_1-u_2
\end{array}
\right)
\approx 
\left(
\begin{array}{l}
	1\\0	
\end{array}
\right)
\quad \mbox{since} \quad
DT^{-1}_{\pm,d}(p,p) 
= 
\left(
\begin{array}{ll}
	\infty& \infty \\ 
	2 & -1 
\end{array}
\right), 
$$ 
where here $\infty$ has to be understood as a limit.
Concretely, the tangent vector of $W_d^s(0)$ at the point $(-p,p)$ is horizontal. 

Moreover, taking into account the symmetry 
\eqref{simetria}, 
the points with highest value of $y$ in $W_d^s(0)$ should be symmetric. Actually, they coincide with the two points $q^{\pm}=\left(q_x^{\pm},q_y^{\pm}\right)$ which are mapped by $T_d$ to a point $q=\left(q_x,q_y\right)\in \Gamma$ whose tangent vector has slope $m=2$.

\section{Proof of Theorem A(${\rm b}$): The case $d$ odd and $a=1$}\label{sec:d_odd_a1}

For the whole section we assume that $d\geq 3$ is odd  and $a=1$. The proof of Theorem A(b) is quite long and therefore we will split it into several lemmas and propositions. Roughly speaking the strategy is as follows. First we will see that $\A_d(0)$ is open, simply connected and that $[-1/2,0]\times \{0\}\subset \A_d(0)$ (Proposition \ref{prop:d_odd_a_1_boundary}). Second we will show that there exists a hyperbolic two-cycle of saddle type whose unstable manifold intersects $[-1/2,0]\times \{0\}$. From this we will show that the two-cycle as well as its stable manifold belong to $\partial \A_d(0)$ (Proposition \ref{prop:inclusio}). And finally we will see that  $\partial \A_d(0)$ is unbounded (Proposition \ref{prop:d_odd_a_1_unbounded}).

Let $b\in(0,1/2]$. We denote by $Q_b\subset \mathbb R^2$ be the compact convex polygon bounded by the straight segments 
\vglue 0.2truecm
\begin{tabular}{ll}
$A_b:=\{(2b^d,y)\in \mathbb R^2\mid\ y\in[0,2b^d] \}$,   & \qquad $B_b:=\{(x,2b^d)\in \mathbb R^2 \mid \ x\in[0,2b^d]\}$, \\
& \\
$C_b:=\{(x,2x+2b^d)\in \mathbb R^2 \mid \ x\in[-b^d,0]\}$, & \qquad $D_b:=\{(-b^d,y)\in \mathbb R^2\mid\ y \in [-b^d,0] \}$, \\
& \\ 
$E_b:=\{(x,-b^d)\in \mathbb R^2\mid \ x \in [-b^d,0]$\}, & \qquad $F_b:=\{(x,\frac{1}{2}x-b^d)\in \mathbb R^2\mid \ x\in[0,2b^d]\}$.
\end{tabular}
\vglue 0.2truecm
\noindent We denote $Q^{\star}:=Q_{1/2}$. 

\begin{proposition}\label{prop:d_odd_a_1_boundary}  We have that  $Q^\star \subset \A_d(0)$. In particular, $[-1/2,0]\times \{0\}\subset \A_d(0)$. Moreover, $\A_d(0)$ is open and simply connected.
\end{proposition}

\begin{proof}
From Lemma \ref{lem:properties_invariant} the origin is asymptotically stable and therefore $\A_d(0)$ is open  (see also Figure \ref{fig:local_behavior}). 
The family $\{Q_b\}_{b\in (0,1/2]}$ is a neighbourhood basis of the origin. 

We claim that $T_d(Q_b)\subset {\Int}(Q_b),\ b\in (0,1/2]$. See Figure \ref{fig:TQ} for a sketch of $Q_b$ and its image.

Assume the claim is true. This implies that $Q^{\star}\subset \A (0)$. Since the image of $[-1/2,0]\times \{0\}$ by $T_d$ is the segment $\{(x,2x)\in \mathbb R^2\mid \ x\in[0,(1/2)^d]\} \subset Q^{\star}$ we conclude that $[-1/2,0]\times \{0\}\subset \A_d(0)$ as desired. Moreover, since there exists an open simply connected neighborhood $Q^\star $ containing $(0,0)$ and contained in $\A_d(0)$, the origin is asymptotically stable (our proof demonstrates again that the origin is asymptotically stable). 
Finally,  
$\A_d(0) = \bigcup _{k\ge 0} T_d^{-k}(Q^\star )$.   
Since $T_d$ is one to one, $T_d^{-k}(Q^\star )$ is also open and simply connected, for all $k$.    
Furthermore, since $T_d^{-k-1}(Q^\star ) \supset T_d^{-k}(Q^\star )$,  
we conclude that $\A_d(0)$  is open and simply connected  as well.

The rest of the proof is devoted to prove the claim. Hereafter we remove from the notation the dependence of the whole construction with respect to the parameter $b$, unless strictly necessary. The proof consists in studying the image of  each side  of the boundary of $Q$ by $T_d$. We will get that the image of the boundary of $Q$ is contained in $\Int   (Q)$ and therefore 
$T_d(Q)  \subset \Int (Q)$.


We denote by $\Gamma_A$ the image of the segment $A$ under $T_d$ and similarly for the other pieces of the boundary. 
Next, we prove that each image is contained in $\Int (Q)$.
See Figure \ref{fig:TQ}.

\noindent {\bf The image  $\Gamma_A=T_d(A)$.} We parametrize $\Gamma_A$ as follows
$$
\Gamma_A= T_d(A)= \left \{ \left(\Psi_1(y):= y-(2b^d+y)^d,\
\Psi_2(y):=y-2(2b^d+y)^d\right) , \, y \in [0, 2b^d] \right \}.
$$
  
We check that $\Gamma_A \subset \Int Q \cap \{y < x\}$.
The condition $y < x$ is equivalent to 
$\Psi_2(y) < \Psi_1(y)$  for  $y \in [0, 2b^d]$ which is clearly true.
 The condition 
$y> \frac12 x -b^d$ is equivalent to 
$\Psi_2(y)> \frac12 \Psi_1(y) -b^d$ for $y \in [0, 2b^d]$ which can be
written as 
$\chi_1 (y):= \frac12 y +b^d -\frac32(2b^d +y)^d >0$ for  $y \in [0, 2b^d]$.  
We have that $\chi_1(0)= b^d (1-\frac32 2^d b^{d^2-d})  >0 $ and
$\chi_1(2b^d ) = 2 b^d (1-\frac34 4^d b^{d^2-d} )>0$ since $ b\in (0, 1/2]$ and $d\ge 3$.
Also $\chi_1'' (y)=  -d(d-1)\frac32(2b^d +y)^{d-2} <0$, therefore   
$\chi_1 (y)>0$.
Finally, $\Psi_1(y) < y \le 2b^d$ and 
$\Psi_2(y) \ge -b^d$. The first claim is immediate. For the second we
consider the auxiliary function  
$ 
\chi_2(y):= y +b^d-2(2b^d+y)^d 
$.
We have $\chi_2(0) = b^d(1-2^{d+1} b^{d^2-d} ) >0$ and 
$\chi_2(2b^d) = 3b^d(1-\frac13 2^{2d+1} b^{d^2-d} ) >0$ since $d\ge 3$. Moreover,
$\chi_2'' (y) =  -2d(d-1)(2b^d +y)^{d-2} <0$ and hence 
$\chi_2 (y) >0$.

\begin{figure}[ht]
    \centering
     \includegraphics[width=0.6\textwidth]{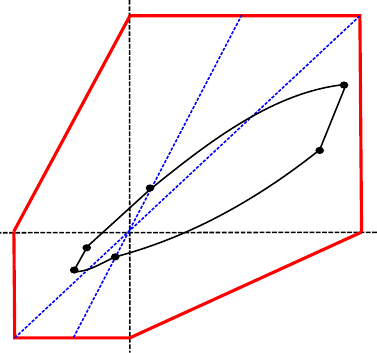}
 \put(-9,90){\scriptsize $\hat a=(2b^d,0)$}
  \put(-17,82){ $\bullet$}
       \put(-147,40){$Q_b$}
     \put(-77,40){$F$}
          \put(-77,107){$\Gamma_A $}
   \put(-9,140){$A$}
             \put(-29,160){$\Gamma_B$}
      \put(-100,240){$B$}
                   \put(-129,155){$\Gamma_C$}
        \put(-230,175){ $C$}
                  \put(-197,99){$\Gamma_D$}
          \put(-267,60){$D$}
            \put(-231,63){$\Gamma_E$}
            \put(-215,0){$E$}
              \put(-205,50){$\Gamma_F $}
 \put(-8,240){\scriptsize $\hat b=(2b^d,2b^d)$}
   \put(-17,235){ $\bullet$}
 \put(-170,240){\scriptsize $\hat c=(0,2b^d)$}
    \put(-180,235){ $\bullet$}
 \put(-310,90){\scriptsize $\hat d=(-b^d,0)$}
   \put(-263,82){ $\bullet$}
  \put(-320,10){\scriptsize $\hat e=(-b^d,-b^d)$}
   \put(-261,9){ $\bullet$}
  \put(-165,3){\scriptsize $\hat f=(0,-b^d)$}
     \put(-179,6){ $\bullet$}
  \put(-57,220){\scriptsize $y=x$}
  \put(-143,200){\scriptsize $y=2x$}
  \put(-185,55){\scriptsize $T_d(\hat a)$}
  \put(-35,140){\scriptsize $T_d(\hat b)$}
  \put(-36,196){\scriptsize $T_d(\hat c)$}
  \put(-180,118){\scriptsize $T_d(\hat d)$}
  \put(-226,75){\scriptsize $T_d(\hat e)$}
  \put(-236,52){\scriptsize $T_d(\hat f)$}
 \caption{\small{Sketch of the closed set $Q_b$ and its image $T_d(Q_b)$ for $d\geq 5$.}}       
    \label{fig:TQ}
    \end{figure}
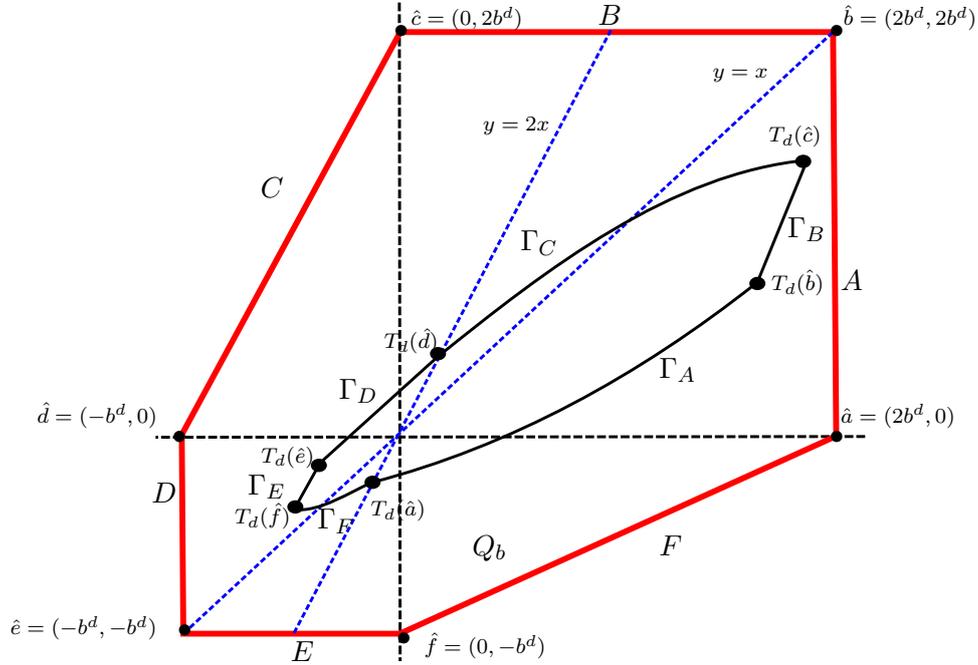
    
\vglue 0.2truecm


\noindent {\bf The image  $\Gamma_B=T_d(B)$.} We parametrize $\Gamma_B$ by $x$: 
$$
\Gamma_B= T_d(B)= \left \{ \left(\Psi_1(x):= 2b^d-(2b^d+x)^d,\ \Psi_2(x):=
2b^d-2(2b^d+x)^d\right) , \, x \in [0, 2b^d] \right \}.
$$
  
It is immediate to see that 
$$
\Psi_1'(x)=- d(2b^d+x)^{d-1}<0 \qquad \text {and } \qquad \Psi_2'(x)=- 2d(2b^d+x)^{d-1}<0.
$$
Therefore, 
$$
\Gamma_B \subset [\Psi_1(2b^d), \Psi_1(0)] \times [\Psi_2(2b^d),\Psi_2(0)] 
\subset (0, 2b^d) \times [0, 2b^d) \subset \Int(Q).
$$

\vglue 0.2truecm


\noindent The image  $\Gamma_C=T_d(C)$. We parametrize $\Gamma_C$ by $x$: 
\[
\Gamma_C = \left \{ \left(\Psi_1(x):=2x+2b^d-(3x+2b^d)^d,\ \Psi_2(x):=2x+2b^d-2(3x+2b^d)^d 
\right) , \, x \in [-b^d,0] \right \}.
\]
  
Similarly as before 
\[
	\Psi_1'(x)  =2-3 d(3x+2b^d)^{d-1}>2-3d2^{d-1}b^{d^2-d}>0 \\
\]
and
\[
	\Psi_2'(x)  =2-6 d(3x+2b^d)^{d-1}>2-6 d2^{d-1}b^{d^2-d}>0.
\]
Then,
$$
\Gamma_C \subset [\Psi_1(-b^d), \Psi_1(0)] \times [\Psi_2(-b^d),\Psi_2(0)] 
\subset (0, 2b^d) \times (0, 2b^d) \subset \Int(Q).
$$
\vglue 0.2truecm

\noindent {\bf The image  $\Gamma_D=T_d(D)$.} We parametrize $\Gamma_D$ by $y$:  Clearly,
$$
\Gamma_D= \left \{ \left( \Psi_1(y):=y-(-b^d+y)^d,\ \Psi_2(y):=y-2(-b^d+y)^d \right) , \,  y\in[-b^d,0] \right \}.
$$
  
First, we check that $\Gamma_D \subset [-b^d, 2b^d]\times [-b^d, 2b^d]$. 
Indeed, 
\[
\Psi'_1(y)=1-d(-b^d+y)^{d-1}>0,\qquad 
\text{and}\qquad 
\Psi'_2(y)=1-2d (-b^d+y)^{d-1} >0 .
\]
Then 
\[
\begin{aligned}
-b^d < \Psi_1(-b^d) \le \Psi_1(y) \le \Psi_1(0) < 2b^d, \\
-b^d < \Psi_2(-b^d) \le \Psi_2(y) \le \Psi_2(0) < 2b^d.
\end{aligned}
\]
The condition $\psi_2(y)< 2\psi_1(y)+2b^d$ reads 
$y-2(-b^d+y)^d < 2(y-(-b^d+y)^d)+2b^d $
which is satisfied if $y +2b^d >0$ which is the case.

The condition $\psi_2(y) > \frac12 \psi_1(y) -b^d $ reads
$y-2(-b^d+y)^d > \frac12 (y-(-b^d+y)^d) -b^d $
which is satisfied if
$ \chi_3(y) := \frac12 y-\frac32(-b^d+y)^d +b^d > 0 $. 
This is indeed true because
$ \chi_3(-b^d)=\frac12 b^d + \frac32(2b^d)^d >0 $ and
$ \chi_3'(y) = \frac12-\frac32 d(-b^d+y)^{d-1} > 0 $
(since $d\ge 3$ and $b\in (0,1/2]$).

\vglue 0.1truecm

\noindent {\bf The image  $\Gamma_E=T_d(E)$.} We parametrize $\Gamma_E$ by $x$: 
$$
\Gamma_E= \left \{ \left( \Psi_1(x):=-b^d-(-b^d+x)^d,\  \Psi_2(x):=-b^d-2(-b^d+x)^d \right) , \,  x\in[-b^d,0] \right \}.
$$
  
In this case we will check that $\Gamma_E\subset (-b^d ,0)\times (-b^d ,0) \subset \Int(Q)$. 
Indeed, directly from the expression of the parametrization we have,  $\Psi_1(x) >-b^d$, $\Psi_2(x)>-b^d$,
$\Psi_1(x)< -b^d-(-2b^d)^d = -b^d(1-2^d b^{d^2-d}) <0$ and    
$\Psi_2(x)<  -b^d(1-2^{d+1} b^{d^2-d}) <0$.

\vglue 0.1truecm

\noindent {\bf The image  $\Gamma_F=T_d(F)$.}  We parametrize $\Gamma_F$ by $x$: 
$$
\Gamma_F= \left \{ \left( \Psi_1(x):=\frac{1}{2}x-b^d-\left(\frac{3}{2}x-b^d\right)^d,
\ \Psi_2(x):=
\frac{1}{2}x-b^d-2\left(\frac{3}{2}x-b^d\right)^d\right) , \,  x\in[0,2b^d] \right \}.
$$

In this case we will also check that 
$\Gamma_F\subset (-b^d ,0)\times (-b^d ,0) \subset \Int(Q)$. 
We have that 
$$
\Psi_1 (0) = -b^d -(-b^d )^d > -b^d, \qquad 
\Psi_1 (2b^d ) = -(2b^d )^d<0 .
$$
Moreover, since
$\Psi'_1 (x) = \frac{1}{2}-\frac32 d\left(\frac{3}{2}x-b^d\right)^{d-1}$
which is positive because
$
\frac32 d\left(\frac{3}{2}x-b^d\right)^{d-1} \le 
\frac32 d 2^{d-1} b^{d(d-1)} \le \frac32 d 2^{-d^2+2d-1)}  \le 9/32<1/2
$
we get $\Psi_1 (x)\in (-b^d ,0)$.

Concerning $\Psi_2$, 
$\Psi_2(0) =-b^d+2\left(b^d\right)^d >-b^d $ and  
$\Psi_2(2b^d)= -2 (2b^d)^d<0$. However,   
it is not (always) monotone. Depending on $b$ and $d$ it may have a maximum at some $x_c\in (0,b^d)$. The value of $x_c$  is obtained from the condition  $\Psi'_2(x_c)=0$. It is the positive solution of
$$
\left(\frac{3}{2}x_c-b^d\right)^{d-1} = \frac{1}{6d}.
$$ 
In case $x_c$ belongs to the interval $(0, 2b^d)$ 
we have that 
$
\Psi_2(x_c) =\frac12 x_c-b^d-\left(\frac{3}{2}x_c-b^d\right)^d
= \frac12 x_c-b^d-\left(\frac{1}{6d} \right)^{d/(d-1)}< -\left(\frac{1}{6d}\right)^{d/(d-1)} <0
$. 
Thus, $\Psi_2(x) <0$ for $x\in [0,2b^d]$. 
This finishes the proof. 
\end{proof}

Following the steps of the strategy of the proof of Theorem A(b) described at the beginning of the section.  We start  checking that $\{p_0=(0,1),p_1=(0,-1)\}$ forms a hyperbolic two-cycle. 

Since $DT_d(p_0) =DT_d(p_1)$, the chain rule implies that
\begin{equation} \label{eq:dif_matrix_two_cycle}
DT_d^2(p_0)=DT_d^2(p_1)= DT_d(p_0) DT_d(p_1)
=
\left (
\begin{array}{ll}
 3d^2-2d & \ \ 3d^2-4d+1 \\
 6d^2-2d & \ \ 6d^2-6d+1 
\end{array}
\right ).
\end{equation}
  
A direct computation shows that the eigenvalues and eigenvectors of $DT_d^2(p_j),\ j=0,1$, are given by 

\begin{equation*}
\lambda^{\pm}_d =  \frac{1}{2}\left ( 9d^2- 8d  + 1 \pm (3d-1)\sqrt{ 9 d^2 -10d + 1} \right) 
\end{equation*}
  
and
\begin{equation}\label{eq:m}
 ( 1,m^{\pm}_d) =   \left(1,\ \frac{ 4d}{1-d \pm \sqrt{ 9d^2-10d +  1}} \right),
\end{equation}
respectively. Finally, it is straightforward to check that 
both eigenvalues are strictly positive. Moreover,  
$\lambda^{-}_d$ is strictly decreasing and  $\lambda^+_d$ is strictly increasing, with respect to the parameter $d$. We also have
\begin{center}
\begin{tabular}{ccl}
$ \lim\limits _{d\to \infty}  \lambda^{-}_d=1/9$ & and & $1/9<\lambda^{-}_d \le   
\lambda^{-}_3 =  29-8\sqrt{13} \approx 0.1556$ \\
$ \lim\limits_{d\to \infty}  \lambda^{+}_d=\infty$ & and & 
 $\lambda^{+}_ d \ge \lambda^{+}_ 3 = 29+8\sqrt{13}\approx  57.8444.$
\end{tabular}
\end{center}

On the other hand, $m^{-}_d$  is negative and  strictly  increasing while  $m^+_d$ is positive and  strictly decreasing (both with  respect to the parameter $d$). We also have

\begin{center}
\begin{tabular}{ccl}
$\lim\limits_{d\to \infty} m^{-}_d =-1$ & and & 
 $-1.3028 \approx \frac{-6}{1+\sqrt{13}} = m^{-}_3  \le m^{-}_d <-1$ \\
 $ \lim\limits _{d\to \infty} m^{+}_d =2$ & and & $2<m^{+}_d \le m^{+}_3 = \frac{6}{\sqrt{13}-1} \approx 2.3028$ 
 \end{tabular}
\end{center}

Therefore,  the two cycle  $\{p_0, p_1\}$ is a  hyperbolic saddle point. In what follows we will denote by $W^s:=W^s_{\{p_0,p_1\}}$ and $W^u:=W^u_{\{p_0,p_1\}}$ the (global) stable and unstable manifolds of the periodic orbit $\{p_0,p_1\}$, respectively. We  split $W^u=W^u_{p_0}\cup W^u_{p_1}$ where $W^u_{p_j}$ is the (global) unstable manifold of the fixed point $p_j$ for the map $T_d^2$, $j=1,2$. Similarly, $W^s=W^s_{p_0}\cup W^s_{p_1}$ for the stable manifold. Consequently, we remark that $W^s$ and $W^u$ refer to the manifolds associated to the hyperbolic periodic orbit $\{p_0,p_1\}$, and hence they are not the manifolds associated to the origin (with a similar notation) studied and considered in Sections  \ref{sec:local} and 
\ref{sec:d_even}.

To simplify the notation, unless strictly necessary, we drop the dependence of $\lambda_d^{\pm}$ and $m_d^{\pm}$ with respect to the parameter $d$. Thus, we will write
$$
\lambda^\pm:=  
\lambda^{\pm}_d \quad \quad \mbox{and} \quad \quad m^{\pm}:=m^{\pm}_d .
$$
We  introduce $ m^{\star}=\frac72$.

Next lemma gives a precise description of the geometry of $W^u_{p_0}$ that we will use to prove that $\{p_0,p_1\} \subset \partial \mathcal A_d(0)$ and finally to prove that  $W^s \subset \partial \mathcal A_d(0)$.

\begin{lemma}\label{lemma:d_odd_a_1_unstable} 
Let $\D$ be the closed triangle determined by the vertices 
$$
p_1=(0,-1), \qquad \left(\frac{1}{ m^+ +1},\frac{-1}{ m^+ +1}\right) \qquad \text {and} \qquad \left(\frac{1}{m^\star+1},\frac{-1}{m^\star+1}\right) .
$$
Then, there is a local piece of $W^u_{p_0}$ (attached to $p_0$) tangent to the line $y=1+ m^+ x$ contained in $T_d(\D)$. Moreover, if we parametrize $W^u_{p_0}\cap \{ y\le 1\}$ as $W^u_{p_0}:=\{ \varphi(t) \mid\ t \geq 0\}$, with $\varphi(0)=(0,1)$ and $\varphi(t)\subset \Int (T_d(\D))$ for $t\in(0,t_0)$ and  $\varphi(t_0) \in \partial T_d(\D)$ then 
$$
\varphi(t_0) \subset \partial T_d(\D) \cap \{y=x\} = \left\{(s,s)\mid\ \frac{1}{ m^+ +1}\le s \le  \frac{1}{m^\star +1} \right \}.
$$  
See  Figure \ref{fig:D_TD} (right).
\end{lemma}

\begin{proof} The triangle $\D$ can also be represented as
\begin{equation*}
\D = \{ (t,-1+mt) \, | \    t \in [0,1/(m+1)], \ m \in [ m^+ ,m^\star]   \}.
\end{equation*}

The proof of this lemma will follow from an accurate description of the sets $T_d(\D)$ and $T_d^{-1}(\D)$, their relative position and geometry in the plane, and the behaviour of the map $T_d^{-2}:T_d(\D) \to T_d^{-1}(\D)$. 
\vglue 0.2truecm
\noindent {\bf The shape of  $T_d(\D)$.} 
 We consider the decomposition of $\D$ into the segments 
\begin{equation} \label{eq:lms}
\ell_m=\{(t,-1+mt),\ t \in [0,1/(m+1)]\}, \qquad \text{with} \qquad m \in [ m^+,m^{\star}].
\end{equation}
If we write $\gamma(t):=\gamma_m(t)=T_d\left(\ell_m\right):=(x_m(t),y_m(t))=:\left(x(t),y(t)\right)$ we have 
\begin{equation} \label{imatgescorbesenTD}
x(t) = mt -1 - ( (m+1)t -1)^d  \qquad \text{and} \qquad  y(t) = mt -1 - 2( (m+1)t -1)^d.
\end{equation}
Thus, the first derivatives of $x(t)$ and $y(t)$ are given by
\[
x'(t) = m- d(m+1) ( (m+1)t -1)^{d-1} \quad \text{and} \quad  y'(t) = m- 2d(m+1)( (m+1)t -1)^{d-1}.
\]
Easy computations show that $x^{\prime}(t)$ and $y^{\prime}(t)$ vanish at the points  
\[
r^{\pm} = \frac{1}{m+1} \left[ 1 \pm 
\left({\frac{m}{d(m+1)} } \right)^{1/(d-1)} \right]  \quad \text{and} \quad s^{\pm} = \frac{1}{m+1} \left[ 1 \pm 
\left({\frac{m}{2d(m+1)} } \right)^{1/(d-1)}  \right],
\]
respectively. Moreover, 
$$
0 < r^{-} < s^{-}<\frac{1}{m+1} < s^+ < r^+ \qquad \text{and} \qquad x\left(\frac{1}{m+1}\right)=y\left(\frac{1}{m+1}\right)=\frac{-1}{m+1},
$$
where $t=1/(m+1)$ corresponds to the common maximum of $x'(t)$ and $y'(t)$. See  Figure \ref{fig:D_TD} (left). In summary, the components $x(t)$ and $y(t)$ of the curve $\gamma(t)$ are polynomial functions in $t$, having a unique minimum in the interval $[0,1/(m+1)]$ located at $t=r^{-}$ and $t=s^{-}$,  respectively, and sharing the same negative value, $-1/(m+1)$, at $t=1/(m+1)$. See the middle picture in Figure 
\ref{fig:D_TD}.

\begin{figure}[]
    \centering
     \includegraphics[width=0.9\textwidth]{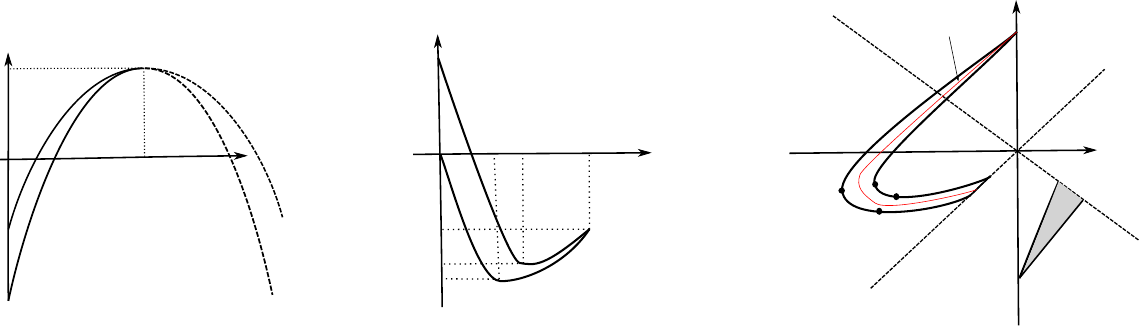}
               \put(-315,72){\tiny $(t,x'(t))$}
               \put(-385,30){\tiny $(t,y'(t))$}
             \put(-393,60){\tiny $r^-$} 
               \put(-378,51){\tiny $s^-$}
          \put(-356,50){\tiny $\frac{1}{m+1}$}
            \put(-265,32){\tiny $\frac{-1}{m+1}$}
             \put(-255,90){\tiny $1$}
                    \put(-404,88){\tiny $m$}
     \put(-199,66){\tiny $\frac{1}{m+1}$}
          \put(-225,64){\tiny $r^-$}
     \put(-214,64){\tiny $s^-$}
     \put(-240,85){\tiny $(t,y(t))$}
          \put(-200,20){\tiny $(t,x(t))$}
     \put(-110,73){\small $T_d(\D)$}
       \put(-140,45){\tiny $\gamma_{ m^+}(r^{-})$}
              \put(-100,33){\tiny $\gamma_{ m^+}(s^{+})$}
     \put(-40,100){\tiny $(0,1)$}
          \put(-9,90){\tiny $y=x$}
          \put(-124,110){\tiny $y=-x$}
          \put(-74,105){\tiny ${W}^u_{p_0}$}
     \put(-26,26){\small $\D$}
          \put(-69,15){\tiny $(0,-1)$}
          \put(-44,15){\tiny $\bullet$}
          \put(-45,100){\tiny $\bullet$}
           \caption{\small{Left figure shows the graphs of the curves $x^{\prime}(t)$ and $y^{\prime}(t)$. The center picture shows the graphs of the curves $x(t)$ and $y(t)$. Finally, the right picture displays a sketch of the triangle $\D$ and its image $T_d(\D)$.}}
    \label{fig:D_TD}
    \end{figure}
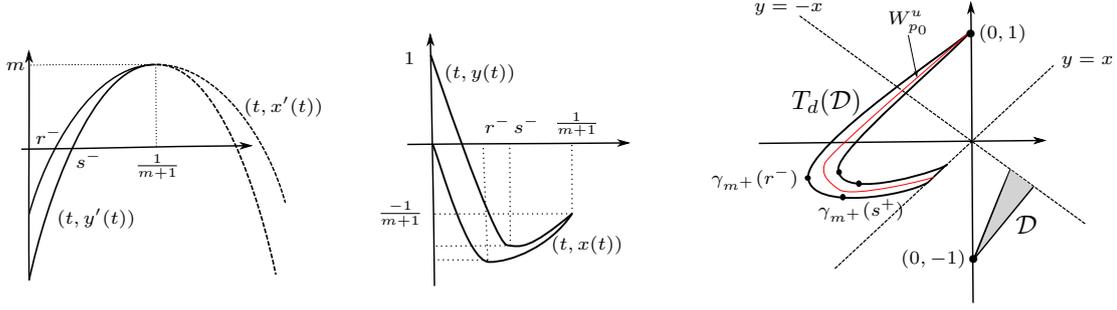

To conclude the description of the shape of $\gamma$, see Figure \ref{fig:D_TD} (right).  Let us prove that its image can be represented as the union of two graphs with respect to the variable $x$ (i.e., it admits a piecewise parametrization of graphs with respect to $x$). We write 
$$
\gamma(x)= (x,\gamma^{(2)}(x)) :=  \left(x,y\left(t(x)\right)\right),
$$
where $t(x):=x^{-1}(t)$ is one of the two branches of the inverse of $x(t)$.
 Some direct computations show that
\begin{equation}\label{eq:derivatives_ym}
\begin{split}
\frac{d\gamma^{(2)}}{dx}&=\frac{dy}{dt}\left(\frac{dx}{dt}\right)^{-1}, \\ 
\frac{d^2\gamma^{(2)}}{dx^2}&=\left(\frac{dx}{dt}\right)^{-3}\left(\frac{d^2y}{dt^2}\frac{dx}{dt}-\frac{dy}{dt}\frac{d^2x}{dt^2}\right)=\left(\frac{dx}{dt}\right)^{-3}\left(d(d-1)m(m+1)^2\left((m+1)t-1\right)\right)^{d-2},
\end{split}
\end{equation}
everything evaluated at the corresponding branch of $t=t(x)$.
We  denote $\gamma^{(2)}_{\rm u}:=\gamma^{(2)}_{{\rm u},m}$ and  $\gamma^{(2)}_{\ell}:= \gamma^{(2)}_{{\ell},m}$ the functions corresponding to the upper (concave) and lower (convex)  graphs,  respectively. 

From the previous discussion, $dx/dt$ has a unique zero at $t=r^{-}$  and it is monotone in the whole interval $[0,1/(m+1)]$, see Figure \ref{fig:D_TD} (left).
For the upper branch, corresponding to $0\leq t <r^{-}$, we have $dx/dt<0$ and therefore $\gamma^{(2)}_{\rm u}(x)$ is increasing and concave (see \eqref{eq:derivatives_ym}) and for the lower branch, $r^{-}< t \le 1/(m+1)$, we have $dx/dt>0$ and therefore $\gamma^{(2)}_{\ell}(x)$ is  convex having a minimum at $x(s^-)$ (see again \eqref{eq:derivatives_ym}). See  Figure \ref{fig:D_TD} (right), we have drawn (qualitatively) the curve $\gamma$ for the values  $m=m^+$ and $m=m^{\star}$. We remark that $\gamma^{(2)}_{\rm u,m^+} (x)$ is tangent at $p_0$ to the line $y=m^+ x+1$ since it is the image of the side of $\D$ tangent to $W^u_{p_1}$. 
Since $T_d$ sends the line $\{y=-x\}$ to $\{y=x\}$ the images of all the curves $\gamma_m$ end up at $\{y=x\}$. 
All together determines the shape of $T_d(\D)$.
Moreover, in the light of the above arguments we have that
$$
\partial 
T_d(\D)=\gamma^{(2)}_{{\rm u},m^+} \cup \gamma^{(2)}_{{\ell},m^+}  \cup \gamma^{(2)}_{{\rm u},m^\star} \cup \gamma^{(2)}_{{\ell},m^\star} \cup \{(x,x) \mid\  \frac {-1}{m^+ +1} < x < \frac{-1}{m^\star+1}\}.
$$ 
Hereafter we will refer to 
$$\gamma^{(2)}_{{\rm u},m^+} \cup \gamma^{(2)}_{{\rm l},m^+} \qquad \text {and} \qquad \gamma^{(2)}_{{\rm u},m^\star} \cup \gamma^{(2)}_{{\rm l},m^\star}
$$
as the left and  right the boundaries of $T_d(\D)$, respectively. See Figure \ref{fig:D_TD} (right) and Figure \ref{fig:D_T_minus_D}.

\vglue 0.2truecm

\noindent {\bf The shape  of $T_d^{-1}(\D)$.}  We consider the same decomposition of $\D$ into the  segments $\ell_m$ as in \eqref{eq:lms}. We denote $\Gamma(t):=\Gamma_m(t)=T_d^{-1}\left(\ell_m\right):=(\alpha_m(t),\beta_m(t))=:(\alpha(t),\beta(t))$. We have 
\begin{equation}\label{eq:alpha_beta}
\alpha(t) = (m-2)t -1 + ( (1-m)t +1)^{1/d} , \quad  \beta(t) = (2-m)t +1,
\quad t\in [0, 1/(m+1)].
\end{equation}
  
Therefore, the first and second derivatives of $\alpha(t)$ and $\beta(t)$ are given by
\begin{equation*}
\begin{split}
&\alpha'(t) = m-2 + \frac{1-m}{d} ((1-m)t +1)^{(1-d)/d},  \ \ \ \qquad  \beta'(t) = 2-m, \\
&\alpha''(t) = \frac{1-d}{d^2} (1-m)^2 ((1-m)t +1)^{(1-2d)/d}<0, \quad \beta''(t)=0 .
\end{split}
\end{equation*}
Clearly $\beta'(t) < 0$ since $m \geq m^+ > 2$. Next, we focus the attention on $\alpha^{\prime}(t)$. The line $t=1/(m-1) $ is a vertical asymptote (outside the domain $(0,1/(m+1))$) and simple computations show that $\alpha'(t)=0$  if and only if $t=t^{\pm}$ where
\[
t^{\pm} :=t^{\pm}_m = \frac{1}{m-1} \pm \left( \frac{m-1}{d^d(m-2)^d}\right)^{1/(d-1)}.
\]
  
Some further computations show that
\begin{equation} \label{eq:t+-}
 t^{-}_{m^+} <0, \qquad t^{-}_{m^{\star}} >\frac{1}{m^{\star}+1}  \qquad \text{and} \qquad t_m^{+} > \frac{1}{m-1}, \qquad \forall m\in [m^+,m^{\star}].
\end{equation}
  
Since $\alpha^{\prime}_{m^+}(0)
= 4(d-1 ) \left(1-d+\sqrt{1-10d+9d^2}\right)^{-1} +1/d -2<0$ and $\alpha^{\prime}_{m^\star}(0)>\frac12 (3-5/d) >0$, it follows from the previous arguments and \eqref{eq:t+-} that $\alpha_{m^+}(t)$ is monotonically decreasing and $\alpha_{m\star}(t)$ is monotonically increasing in the considered domain. Consequently, $\Gamma_{m^+}$ and 
$\Gamma_{m^\star}$ can be expressed as graphs of  monotone functions of the form
$$
\Gamma(x)=(x,\Gamma^{(2)}(x)):= \left(x,\beta\left(t(x)\right)\right),
$$
where $t(x):=\alpha^{-1}(t)$ for $m=m^+$ and $m=m^{\star}$, respectively. 
We have
\begin{equation}
	\label{derivadesalphabeta}
\begin{split}
\frac{d\Gamma^{(2)}}{dx}&=\frac{d\beta}{dt}\left(\frac{d\alpha}{dt}\right)^{-1},  \\ 
\frac{d^2\Gamma^{(2)}}{dx^2}&=-\left(\frac{d\alpha}{dt}\right)^{-3}
\frac{d\beta}{dt}\frac{d^2\alpha}{dt^2}=-\left(\frac{d\alpha}{dt}\right)^{-3}\left(
\frac{d-1}{d^2}
\frac{(m-2)(m-1)^2}{\left(  (1-m)  t+1 \right)^{(2d-1)/d}}
\right),
\end{split}
\end{equation}
everything evaluated at $t=t(x)$. 
Indeed, taking into account \eqref{derivadesalphabeta}, when $m=m^+$, $\Gamma^{(2)}$ is increasing and convex, while when $m=m^\star$, $\Gamma^{(2)}$ is decreasing and concave. From \eqref{eq:alpha_beta} we conclude that 
$$
\beta(t) > \beta (1/(m+1)) = 3/(m+1)> 3/(m^\star+1)=2/3,
$$
and then the preimage $T^{-1}(\D) $ is above the line $\{y =2/3\}$. Finally we notice that the image by $T^{-1}_d$ of the segment $\D\cap \{y=-x\}$ is contained in (the graph of) 
$$
x=\phi(y):=-y + \left(\frac23 y \right)^{1/d}.
$$
Then 
$$
\phi'(y)=-1 +\frac{2}{3d} \left(\frac23 y\right)^{(1-d)/d} < -1+\frac{3}{2d}<  -\frac12,   
$$
and the function $y = \phi^{-1} (x) $ is decreasing in the corresponding domain.

In particular the curve $\Gamma_{m^+}(x)$ belongs to the second quadrant, while $\Gamma_{m^\star}(x)$ belongs to the first one. In Figure \ref{fig:D_T_minus_D} we display $T_d(\D)$ and $T_d^{-1}(\D)$ and we can see the relative position of these two sets and the initial triangle $\D$. We emphasize that, from  arguments above, $\gamma_{m^+}(x)$ and $\Gamma_{m^+}(x)$ are both tangent to the line $y=m^+ x+1$, but, using the convexity properties,  $\gamma_{m^+}(x)$ is below this line while $\Gamma_{m^+}(x)$ is above it, hence their relative position illustrated in Figure \ref{fig:D_T_minus_D} is the right one.

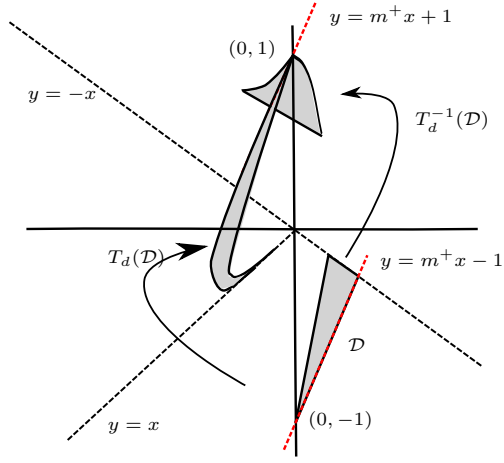
\begin{figure}[ht]
    \centering
     \includegraphics[width=0.4\textwidth]{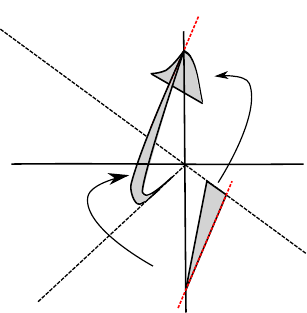}
     \put(-140,73){\tiny $T_d(\D)$}
          \put(-25,125){\tiny $T_d^{-1}(\D)$}
     \put(-38,72){\tiny $y=m^+x-1$}
          \put(-57,164){\tiny $y=m^+ x+1$}
    \put(-95,152){\tiny$(0,1)$}
       \put(-170,135){\tiny $y=-x$}
      \put(-140,10){\tiny $y=x$}
     \put(-50,40){\tiny $\D$}
         \put(-65,12){\tiny $(0,-1)$}
           \caption{\small{A qualitative representation of the triangle $\D$ and the sets $T_d(\D)$ and $T_d^{-1}(\D)$.}}
    \label{fig:D_T_minus_D}
    \end{figure}

Consider now $T_d^{-2}:T_d(\D)\to T_d^{-1}(\D)$. From the stable/unstable manifold theorem and the relative position and geometry of the sets $T_d(\D)$ and $T_d^{-1}(\D)$, we can conclude that, locally, $W^{u}_{p_0}$ exists, it is tangent to $y=m^+x+1$, and it is contained in $T_d(\D)$. In particular we also conclude that there is a local piece of 
$W^{u}_{p_1}$ attached to $p_1$ belonging to $\D$.

We claim that ${W}^u_{p_0}$ may only {\it leave}  $T_d(\D)$ through the piece of the boundary given by $\partial T_d(\D) \cap \{y=x\}$ (later we will see that $W^u_{p_0}$  does leave  $T_d(\D)$ through this boundary). To check the claim we first observe that   $W^{u}_{p_0}\cap \{y\le 1\}$ can be parametrized $W^u_{p_0}:=\{\varphi(t)\mid\   t \geq 0 \}$, with $\varphi(0)=(0,1)$ (see \cite{ParametrizationMethod1}). Second, we suppose it leaves $T_d(\D)$ either for the left or the right boundaries of $T_d(\D)$, and we get a contradiction.

Let $p=\varphi(t_0)$ for some $t_0>0$ such that $\{\varphi(t),\ t\in (0,t_0)\} \subset \Int (T_d(\D))$ for all $t\in (0,t_0)$ and $p\in \partial T_d(\D) \setminus \{x=y\}$ (that is, it leaves $T_d(\D)$ through the left or right boundaries of it). Consider $q:=T_d^{-2}(p)$. Since $W^u_{p_0}$ is invariant by $T_d^{-2}$ we have that $q=\varphi(t_q)$ for some $t_q\in (0,t_p)$. However, we also have 
$$
q\in \partial \left(T_d^{-1}(\D)\right)\setminus T_d^{-1}\left(\partial \D \cap \{y=-x\}\right),$$
which provides a contradiction (see Figure \ref{fig:D_T_minus_D}).
\end{proof}

Let $\mathcal E$ be the closed triangle determined by the vertices $\tau_0=(0,1)$, $\tau_1=(-1/2,0)$ and $\tau_2=(-1/3,0)$. 

Next two lemmas refer to the set $T_d(\D)\cap \{y\geq 0\}$. First we show that this set belongs to the triangular region $\mathcal E$  and, second,  we will give relevant information of the dynamics of $T^{-2}|_{\mathcal E}$ (and therefore in $T_d(\D)\cap \{y\geq 0\}$). All together implies two main properties for $W^u_{p_0}$. On the one hand,  we will prove that $W^u_{p_0} \cap [-1/2,0]\times \{0\} \ne \emptyset$ and, on the other hand,  we will prove that $W^u_{p_0} \cap \{y=x\} \ne \emptyset$ (such intersection happens on the fourth quadrant). See Figure \ref{fig:mathcalC}.

\begin{lemma} \label{lem_c}
We have
$$
T_d(\D)\cap \{y\geq 0\} \subset \mathcal E.
$$
\end{lemma}
\begin{proof}
We will check that the left and the right boundaries of  
$T_d(\D)\cap \{y\geq 0\}$, given by pieces of the curves $\gamma_{m^+}$ and 
$\gamma_{m^\star}$, respectively, are contained in $\mathcal E$. By the discussion we did when analyzing the shape of $T_d(\D)$ we know that both curves can be written as graphs of concave functions that only intersect at the point $p_0=(0,1)$ (see Lemma \ref{lemma:d_odd_a_1_unstable}). Moreover the slope $s_{m^+}$ of $\gamma_{m^+}$ at $t=0$ is $m^+ >2$ (the slope of the left side of $\mathcal E$). Indeed,
$$
2<s_{m^+} \le 6/(\sqrt{13}-1) \lessapprox 2.3028 .
$$
All together implies that the statement of the lemma follows from proving that the (only) intersection of $\gamma_{m^\star}$ with $y=0$ happens at a point $x<-1/3$.

Consider the second component $y(t)$ of  $\gamma_{m^{\star}}$ defined for 
$t\in [0,1/(m^\star + 1)]$.
Since $y(0)=1$, $y(1/(m^\star + 1))=
-1/(m^\star + 1)<0$ and $y''(t) >0$ there exists a unique $t_1\in  [0,1/(m^\star + 1)] $ such that 
$y(t_1)=0$. See Figure \ref{fig:D_TD} (center). 

When $d=3$ we can localize  $t_1$ with some precision. Let 
$t_1^-= 1/18$ and $t_1^+= 1/16$. Both values belong to $[0, r^-]$ where the functions $x(t)$ and $y(t)$ are decreasing. See \eqref{imatgescorbesenTD}. We have
$$
y(t_1^-)= -\frac{29}{36}+ 2 \Big(\frac{27}{36}\Big)^3 > \frac{1}{30} 
\qquad \text{and}
 \qquad 
y(t_1^+)= -\frac{25}{32}+ 2 \Big(\frac{23}{32}\Big)^3 < \frac{-1}{30}.
$$
This means that $t_1\in (1/18, 1/16)$ and that, since $\gamma_{m^{\star}}$ is the graph of a concave function,
$
x(t_1) < x(t_1^-) = -\frac{29}{36}+  \Big(\frac{27}{36}\Big)^3 < -\frac{1}{3}
$.
By concavity we get that $\gamma_{m^\star}\cap \{y\ge 0\}
\subset \mathcal E$.

To deal with the general value of $d$ odd, $d\ge 3$, we will see that if we consider the intersection point $(x(t_1), 0)$ as  a function of $d$, it is decreasing so that for $d\ge 3$, $x(t_1) < -\frac{1}{3} $.
Indeed, we compute the derivative of $x(t_1)$ with respect to $d$.

We write $y=y(t,d)$.
Let $t_1(d) $ be the parameter such that $y(t_1(d),d)=0$. 
We want to compute $(x(t_1(d),d))'$, where prime stands for the derivative with respect to $d$.

Derivating impliticly (and simplifying notation) we have
$$
(t_1(d))' = -\frac{\partial y}{\partial  d} (t_1(d),d) /  \frac{\partial y}{\partial  t} (t_1(d),d)=:-\left(\frac{\partial y}{\partial  d}/\frac{\partial y}{\partial  t}\right)|_{(t_1(d),d)}=:-\left(\frac{\partial y}{\partial  d}/\frac{\partial y}{\partial  t}\right).
$$
Then, 
\begin{align*}
(x(t_1(d),d))' & = \frac{\partial x}{\partial  t} \times (t_1(d))'
+ \frac{\partial x}{\partial  d}  \\
&  =
\frac{\partial x}{\partial  t}\times \Big(-
 \frac{\partial y}{\partial  d} / \frac{\partial y}{\partial  t}\Big) + \frac{\partial x}{\partial  d} 
 \\
& = 
\Big(1/ \frac{\partial y}{\partial  t} \Big) \times
\Big(
\frac{\partial x}{\partial  d} \times \frac{\partial y}{\partial  t}
-\frac{\partial x}{\partial  t} \times \frac{\partial y}{\partial  d}
\Big).
\end{align*}
Taking the corresponding derivatives from equation \eqref{imatgescorbesenTD} and simplifying we get 
$$
(x(t_1(d),d))' =-\Big(1/ \frac{\partial y}{\partial  t} \Big) 
m^{\star}(1-(m^{\star}+1)t)^d \log(1-(m^{\star}+1)t) <0.
$$
Indeed, we are evaluating the right side of the above equation at the point $(t_1(d),d)$ with $0<t_1(d)<r^{-}(d)$. Thus, we have 
$$
\left(1/ \frac{\partial y}{\partial  t}\right)|_{(t_1(d),d)} < 0 \qquad \text{and} \qquad 0<1-(m^{\star}+1)t <1. 
$$
\end{proof}

Next lemma tell us that, while the iterates by $T_d^{-2}$  remain
in $\mathcal E$, the sequence of their second coordinates of them is strictly increasing. See Figure \ref{fig:mathcalC}.

\begin{lemma}\label{lem:ming}
Let $(f(x,y),g(x,y)):=T_d^{-2}(x,y)$.   If $(x,y)\in \mathcal E$ then $g(x,y) \geq y$ and the equality only holds when $(x,y)=(0,1)$. 
\end{lemma}

\vglue 0.1truecm

\begin{proof}
 From \eqref{eq:T_inverse_odd} we have that $
g(x,y)=3y-6x+2(x-y)^{1/d} $,
and then 
$g(x,y)\ge y$ in $\mathcal E$ if and only if 
$$
G(x,y):=2y-6x+2(x-y)^{1/d}>0, \qquad \forall (x,y)\in\mathcal E \setminus \{(0,1)\}.
$$
  
To prove this inequality we will show that $G$ restricted to $\mathcal E$ has a global minimum
$G=0$  at $(0,0)$ which is only attained at $(0,0)$.
A direct computation shows that the partial derivatives of $G$ cannot vanish simultaneously, therefore the minimum has to be attained at the boundary of $\mathcal E$. 
It is clear that  the restriction of the function $G$ on each of the three segments of $\partial \mathcal E$ is given by
\begin{equation*}
\begin{split}
\chi_1(x):=
&G(x,0)=2\left(-3x+x^{1/d}\right), \quad \quad\quad\quad\quad\quad\quad   \  x\in [-1/2,-1/3],\\
\chi_2(x):=&G(x,2x+1)=2\left(1-x-(1+x)^{1/d}\right), \quad \quad x\in [-1/2,0], \\
\chi_3(x):=&G(x,3x+1)=2\left(1-(1+2x)^{1/d}\right), \quad \quad\quad\ x\in [-1/3,0].
\end{split}
\end{equation*}
  
Using elementary methods we can check that indeed $\chi_1(x)>0$, $\chi_2(x)\ge 0$ and 
$\chi_3(x)\ge 0$ in the indicated intervals and that $\chi_2(x)= 0$, 
$\chi_3(x)= 0$ only hold when $x=0$. 
\end{proof}

\begin{lemma} \label{lem:w-u_intersections}
The unstable manifold $W^u_{p_0}$ crosses the interval $I_0:=T_d(\D) \cap \{y=x\}$ at some point $(\wh  p, \wh  p)$ such that 
$ \frac {1}{m^+    +1} < \wh  p < \frac{1}{m^\star+1}$.
Moreover, the piece of  $W^u_{p_0}$ from $(0,0)$ to $(\wh  p, \wh  p)$ is contained in $T_d(\D)$. We also have that this piece of $W^u_{p_0}$ cuts the segment
$(-1/2, -1/3) \times \{0\} \subset \R^2$. 
\end{lemma}

\begin{proof}
We first prove the existence of the point $(\wh  p, \wh  p)$. 
A completely analogous procedure will be used in the proof of Proposition \ref{prop:d_odd_a_1_unbounded} and in Section \ref{sec:d_odd_aminus1}.
Let 
\begin{equation*}
	I_0:= T_d(\D) \cap \{y=x\} = \left\{(s,s)\mid\ \frac{1}{m^+ v+1}\le s \le  \frac{1}{m^\star +1} \right \}. 
\end{equation*}
The image $T^{-2}_d (I_0)$ is a curve, which is a piece of the boundary of 
$T_d^{-1}(\D)$ that, be previous arguments, has to cross the left and right boundaries of $T_d(\D)$. Actually, it can be parametrized as
$s \mapsto (-3s+(2s)^{1/d}, 3s)$. In the study of the shape  of $T_d^{-1}(\D)$ we have seen that $T_d^{-1}(\D) \subset \{y> \frac{1}{m^\star+1} = \frac23  \}$.
We define 
\begin{equation*}
I_1 = T_d^2(T_d^{-2}\left(I_0\right)\cap T_d(\D))\subset I_0
\end{equation*}  
and, in general,
\begin{equation*}
	I_n = T_d^{2n}(T_d^{-2n}\left(I_{n-1}\right)\cap T_d(\D)), \qquad n\ge 1.
\end{equation*}  
It is clear that 
$	I_n  \subset I_{n-1}$ for all $n\ge 1$. 

Then, $I_\infty=\cap_{n\ge 0} I_n $ is compact and contains the points in 
$T_d(\D)$ such that all their negative iterates by $T_d^{2}$ are in   
$T_d(\D)$.
Moreover, by Lemma \ref{lem:ming}, the sequence of the second components of these iterates is increasing and has to converge to 1.
Then, those points must belong to $W^u_{p_0}$ and therefore there exists $
(\wh  p, \wh  p) \in I_0$ such that 
$$
(\wh  p, \wh  p)\subset W^u_{p_0} \cap \{y=x\}\subset T_d(\D) \cap \{y=x\}\ne \emptyset.
$$
From Lemma \ref{lemma:d_odd_a_1_unstable} the piece of $W^u_{p_0}$ from $(0,0)$ to $(\wh  p, \wh  p)$ must be contained in $T_d(\D)$. Hence Lemma \ref{lem_c} implies that $W^u_{p_0}$ cuts the segment
$(-1/2, -1/3) \times \{0\} \subset \R^2$. 
\end{proof}

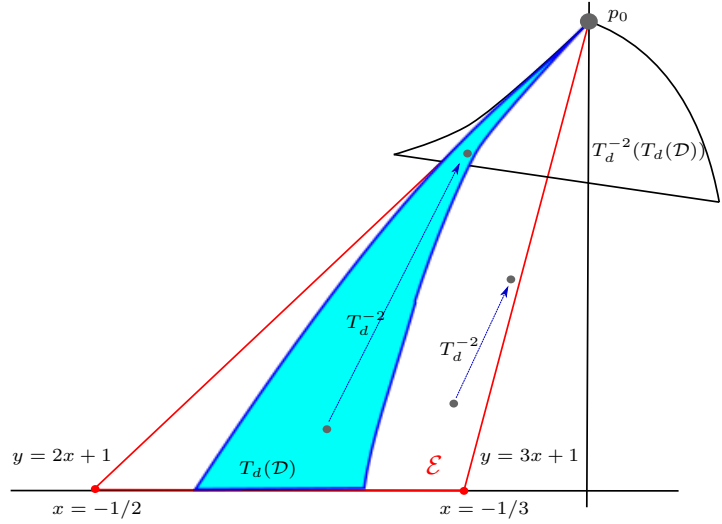
\begin{figure}[ht]
    \centering
     \includegraphics[width=0.6\textwidth]{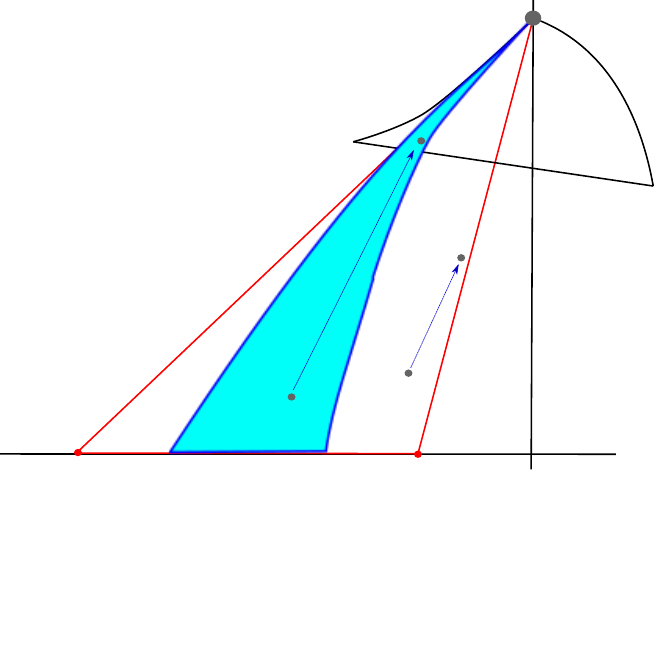}
     \put(-265,90){\tiny $y=2x+1$}
     \put(-90,90){\tiny $y=3x+1$}
     \put(-105,70){\tiny $x=-1/3$}
     \put(-250,70){\tiny $x=-1/2$}
     \put(-42,257){\tiny $p_0$}
       \put(-140,140){\tiny $T_d^{-2}$}
       \put(-105,130){\tiny $T_d^{-2}$}
       \put(-48,205){\tiny $T_d^{-2}(T_d(\D))$}
       \put(-180,85){\tiny $T_d(\D)$}
        \put(-110,85){\small $\r \mathcal E$}
     \caption{\small{$T_d(\D) \cap \{y\geq 0\} \subset \mathcal E$. Points in $\mathcal E$ are mapped ``up" by $T_d^{-2}$.}}
    \label{fig:mathcalC}
    \end{figure}

Next propositions are devoted to show the two further properties of $\A_d(0)$ claimed in Theorem A(b). First we show that the stable manifold of the periodic orbit belongs to $\partial \A_d(0)$ and second we show that $\A_d(0)$ is unbounded, which follows because the stable manifold is unbounded.

\begin{proposition}\label{prop:inclusio} 
	Let $d\geq 3$ odd. Then, $W^s \subset \partial \A_d(0)$. 
\end{proposition}

\begin{proof}
	From Proposition \ref{prop:d_odd_a_1_boundary} we know that $W^u_{p_0}$ crosses the interval $(-1/2,-1/3)\times\{0\} \subset \mathbb R^2$. Let $(x_0,0)\in W^u_{p_0}\cap (-1/2,-1/3)\times\{0\}$
	be the first intersection point of $W^u_{p_0}$ with the segment. From Proposition \ref{prop:d_odd_a_1_boundary}  we also have that $[x_0,0]\times \{0\}\subset \A_d(0)$. We have (recall that $T^{-1}_d(p_0)=p_1 $ and that when $d$ is odd $T_d$ is one-to-one)
	$$
	\displaystyle
	\bigcup_{n=0}^{\infty} T_d^{-n}\left([x_0,0]\times\{0\}\right) \subset \A_d(0) \qquad \text {and} \qquad p_j \in {\rm Acc} \left(\{T_d^{-n}(x_0)\}_{n\geq 0}\right),\quad j=0,1,
	$$ 
	where ${\rm Acc}(X)$ denotes the set of accumulation points of $X$. Since, of course, $\{p_0,p_1\} \not\in \A_d(0)$ we conclude that $p_j\in \partial  \A_d(0),\ j=0,1$. Now, let $q$ be any point in $W^s_{p_0}$ and $U$ a small disc centered at $q$ and let $\Sigma\subset U$ be a transversal segment to $W^s_{p_0}$ through $q$. On the one hand, $W^s_{p_0} \cap U$ is not contained in $\A_d(0)$. On the other hand, by the $\lambda$-Lemma \cite{PALIS69}, the iterates by $T_d^2$ of the points in $\Sigma$ (close enough to $W^s_{p_0}$) accumulate to  $W^u_{p_0}$. Therefore, we would have 
	$$
	(x_0,0) \in {\rm Acc} \left(\{T_d^{n}(\Sigma)\}_{n\geq 0}\right).
	$$
	Since $(x_0,0)\in \A_d(0)$ and $\A_d(0)$ is open we conclude that $U$ contains points of $ \A_d(0)$. If $q\in 
	W^s_{p_1}$ then $T_d(q)\in W^s_{p_0}$ and the conclusion is the same. All together implies that $q\in \partial \A_d(0)$, as desired.
\end{proof}

\begin{proposition}\label{prop:d_odd_a_1_unbounded} $ \A_d(0)$ is unbounded.
\end{proposition}

\begin{proof}

From the previous lemma it is enough to see that the stable manifold  $W^s$ of the hyperbolic two-cycle $\{p_0,p_1\}$ is unbounded. We start  introducing some notation. See Figure \ref{fig:unbounded}.  Let $Q_2^\star$ and $Q_4^\star$ be the closed unbounded subsets of the second and fourth quadrant defined as follows.
$$
Q_2^\star:=\{(x,y)\mid \, x\leq 0,\, y\geq 1\} \quad \mbox{and} \quad Q_4^\star:=\{(x,y)\mid \, x\geq 0,\, y\leq -1\}. 
$$
Next we split the above sets into three pieces. Concretely,
$$
Q_2^\star =\bigcup_{j=1}^3 E_j \qquad \mbox{and} \qquad  Q_4^\star =\bigcup_{j=1}^3 D_j , 
$$
where
\begin{center}
\begin{tabular}{lll}
$E_1=\{x\leq \left(\frac{y+1}{2}\right)^{\frac{1}{d}}-y\mid \ y \geq 1\}$, & \hglue 3truecm & $D_1=\{0\leq x\leq y^{\frac{1}{d}}-y\mid \ y \leq -1\}$, \\
$E_2=\{\left(\frac{y+1}{2}\right)^{\frac{1}{d}}-y \leq x\leq  y^{\frac{1}{d}}-y \mid \ y \geq 1\}$, &  \hglue 3truecm & $D_2=\{y^{\frac{1}{d}}-y \leq x\leq \left(\frac{y-1}{2}\right)^{\frac{1}{d}}-y\mid  \ y \leq -1\}$, \\
$E_3=\{y^{\frac{1}{d}}-y\leq x\leq 0\mid  \ y \geq 1\}$, & \hglue 3truecm  & $D_3=\{x\geq \left(\frac{y-1}{2}\right)^{\frac{1}{d}}-y\mid  \ y \leq -1\}$.
\end{tabular} 
\end{center}
We denote by $\{\mathcal J_\ell,\mathcal I_\ell\}$ with $\ell=1,2$  the straight boundaries of the above sets. That is,
\begin{center}
\begin{tabular}{lll}
$\mathcal J_1=\{(x,1)\mid  \ x\leq 0 \},$ & \hglue 6truecm & \quad $\mathcal I_1=\{(x,-1)\mid  \ x\geq 0 \}$, \\
$\mathcal J_2=\{(0,y)\mid  \ y\geq 1 \},$ &  \hglue 6truecm & $\quad \mathcal I_2=\{(0,y)\mid  \ y\leq -1 \}$.
\end{tabular} 
\end{center}
Finally, we denote by $\{\gamma^{\pm},\sigma^{\pm}\}$ the other boundaries of the sets $E_j$ and $D_j$. That is,
{\small 
\begin{equation} \label{eq:gammas}
\begin{split}
\gamma^{+}&= E_2\cap E_3 = \left\{\left(y^{\frac{1}{d}}-y,y\right) \mid \ y\geq 1\right\}=\{y=(x+y)^d\mid \ x\leq 0,\ y\geq 1\},  \\
\gamma^{-}&=E_1\cap E_2= \left\{\Big(\Big(\frac{y+1}{2}\Big)^{\frac{1}{d}}-y,y\Big)\mid \ y\geq 1\right\}=\{y=-1+2(x+y)^d\mid \ x\leq 0,\ y\geq 1\},    \\
\sigma^{-}&=D_1\cap D_2 =\Big\{\Big(y^{\frac{1}{d}}-y,y\Big)\mid \ y\leq -1\Big\}= \{y=(x+y)^d\mid \ x\geq 0,\ y\leq -1\},   \\
\sigma^{+}&=D_2\cap D_3=\Big\{\Big(\Big(\frac{y-1}{2}\Big)^{\frac{1}{d}}-y,y\Big)\mid \ y\leq -1\Big\} = \{y=1+2(x+y)^d\mid \ x\geq 0,\ y\leq -1\}.
\end{split}
\end{equation} 
}

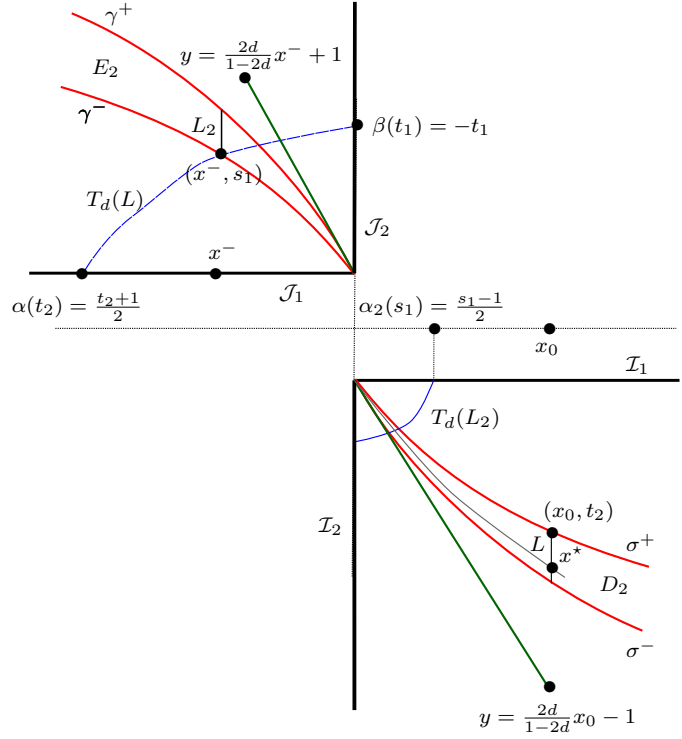
\begin{figure}[ht]
    \centering
     \includegraphics[width=0.55\textwidth]{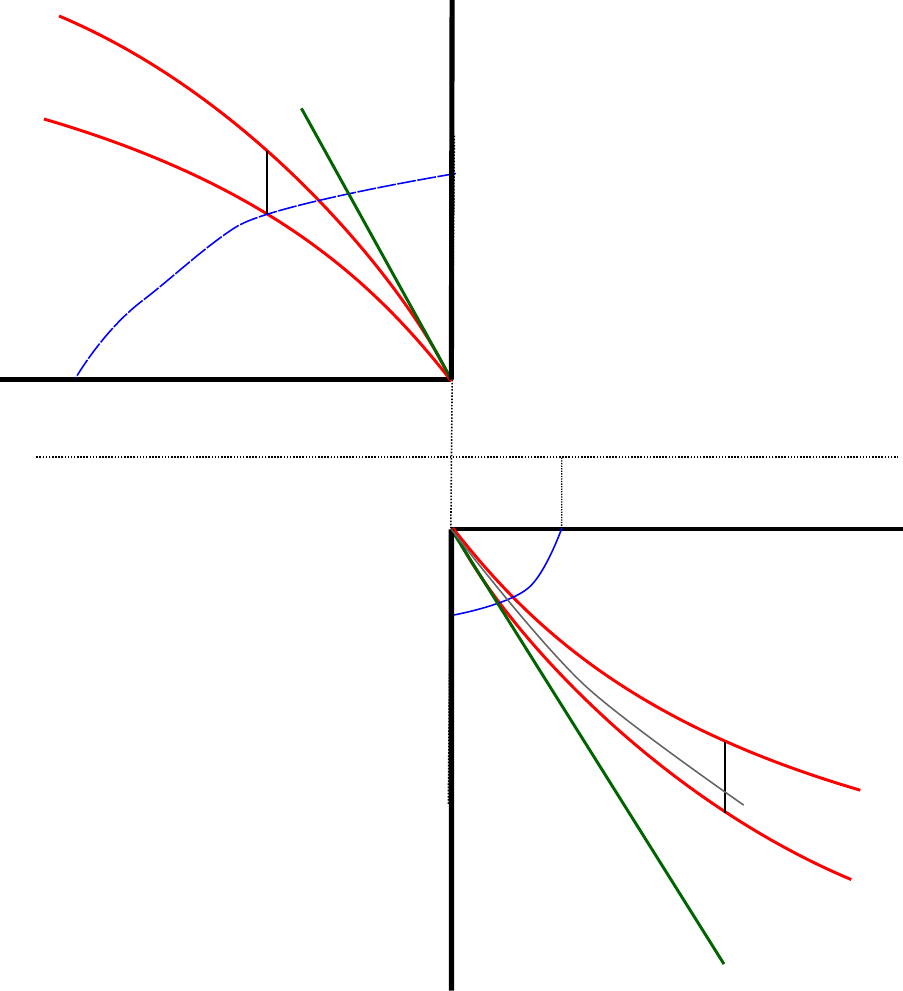}
   \put(-20,60){\scriptsize $\sigma^{+}$}
   \put(-20,20){\scriptsize  $\sigma^{-}$}
   \put(-150,155){\scriptsize $\mathcal J_1$}
   \put(-118,180){\scriptsize $\mathcal J_2$}
   \put(-135,67){\scriptsize $\mathcal I_2$}
   \put(-20,128){\scriptsize $\mathcal I_1$}
   \put(-215,260){\scriptsize $\gamma^{+}$}
   \put(-225,223){\scriptsize  $\gamma^{-}$}
   \put(-56,60){\scriptsize  $L$}
   \put(-45,56){\scriptsize  $x^{\star}$}
   \put(-30,45){\scriptsize  $D_2$}
   \put(-183,218){\scriptsize  $L_2$}
    \put(-220,240){\scriptsize  $E_2$}
   \put(-222,190){\scriptsize  $T_d(L)$}
   \put(-93,108){\scriptsize  $T_d(L_2)$}
    \put(-54,64){ $\bullet$}
    \put(-55,6){ $\bullet$}
    \put(-55,141){ $\bullet$}
    \put(-230,162){ $\bullet$}
     \put(-127,218){ $\bullet$}
      \put(-180,162){ $\bullet$}
      \put(-178,207){ $\bullet$}
     \put(-98,141){ $\bullet$}
      \put(-169,236){ $\bullet$}
      \put(-54,51){ $\bullet$}
     \put(-180,170){ \scriptsize $x^{-}$}
     \put(-189,200){ \scriptsize $(x^{-},s_1)$}
     \put(-115,218){\scriptsize $\beta(t_1)=-t_1$}
     \put(-54,135){\scriptsize $x_0$}
     \put(-51,73){\scriptsize  $(x_0,t_2)$}
      \put(-250,150){\scriptsize  $\alpha(t_2)=\frac{t_2+1}{2}$}
          \put(-120,150){\scriptsize  $\alpha_2(s_1)=\frac{s_1-1}{2}$}
   \put(-75,-5){\scriptsize  $y=\frac{2d}{1-2d}x_0-1$}
   \put(-187,245){\scriptsize  $y=\frac{2d}{1-2d}x^{-}+1$}
   \put(-225,223){\scriptsize  $\gamma^{-}$}
   \put(-225,223){\scriptsize  $\gamma^{-}$}
  \caption{\small{The sets and curves used in the proof of Proposition \ref{prop:d_odd_a_1_unbounded}. }}       
    \label{fig:unbounded}
    \end{figure}
    
\vglue 0.2truecm

See Figure \ref{fig:unbounded} for a qualitative picture and the relative position of all curves and sets. One can check that by construction we have  
$T_d(\gamma^{-})=\mathcal I_1$, $T_d(\gamma^{+})=\mathcal I_2$, $T_d(\sigma^{-})=\mathcal J_2$ and $T_d(\sigma^{+})=\mathcal J_1$.
Consequently,  
$$
T_d(D_2)=\bigcup_{j=1}^3 E_j \qquad \text{and} \qquad T_d(E_2)=\bigcup_{j=1}^3 D_j.
$$
We also notice that  the curves $\gamma^{\pm}(y)$ and $\sigma^{\pm}(y)$ are graphs of  monotonically decreasing functions of $y$. Indeed, for instance, if we write 
$\gamma^\pm=\{\left(\gamma_1^\pm (y),y\right)\mid\ y\ge 1\}$ then
$$
\frac{d\gamma_1^+}{dy}(y)=\frac{1}{d}\left(\frac{1}{y}\right)^{\frac{d-1}{d}}-1\le \frac{1}{d}-1 <0 \qquad \mbox{and} \qquad \frac{d\gamma_1^-}{dy}(y)=\frac{1}{2d}\left(\frac{2}{y+1}\right)^{\frac{d-1}{d}}-1\le \frac{1}{2d}-1 <0.
$$
Let 
\[
\Omega:=T_d^{-1}(E_2).
\]
According to the previous discussion it is clear that  $\Omega \subset D_2$ (remember that $\partial E_2 = \gamma^+\cup \gamma^{-}$). We also claim that $\partial \Omega$ is given by two curves contained in  $D_2$ which can be written as graphs of monotone functions (of $y$ as well as of $x$).  Of course $\partial \Omega = T_d^{-1}\left(\gamma^{-}\right) \cup T_d^{-1}\left(\gamma^{+}\right)$. Using \eqref 
{eq:T_inverse_odd} and \eqref{eq:gammas} we have 
$$
T_d^{-1}\left(\gamma^{-}(y)\right)=
\left(\begin{array}{c}
\xi_1(y) \\
\xi_2(y)
\end{array}
\right)
:=
\left(\begin{array}{c}
3y-2^{\frac{d-1}{d}}\left(y+1\right)^{\frac{1}{d}}+\left[2^{-\frac{1}{d}}\left(y+1\right)^{\frac{1}{d}}-2y\right]^{\frac{1}{d}} \\ 
-3y+2^{\frac{d-1}{d}}(y+1)^{\frac{1}{d}}
\end{array}
\right), \qquad y\ge 1.
$$
  
Thus, we have
\begin{equation*}
\begin{split}
\frac{d\xi_1}{dy}(y)&=3-\frac{1}{d} \left(\frac{y+1}{2}\right)^{\frac{1-d}{d}}+\frac{1}{d}\left[ \left(\frac{y+1}{2}\right)^{\frac{1}{d}}-2y\right]^{\frac{1-d}{d}} \left(\frac{1}{2d}\left(\frac{y+1}{2}\right)^{\frac{1-d}{d}}-2\right)\\
&\geq 3 - \frac{1}{d} > 0
\end{split}
\end{equation*}
  
and 
\begin{equation*}
\frac{d\xi_2}{dy}(y)=-3+\frac{1}{d}\left(\frac{y+1}{2}\right)^{\frac{1-d}{d}} \leq -3 + \frac{1}{d} < 0.
\end{equation*}
Therefore, using the same formulas as the ones in \eqref{derivadesalphabeta}, $T_d^{-1}\left(\gamma^{-}\right)$ can be written as a graph of a monotonically decreasing function  (with respect to $y$ as well as $x$). Similar computations lead to the same conclusion for $T_d^{-1}\left(\gamma^{+}\right)$.

\vglue 0.2truecm

\noindent {\bf Claim 1.} {\it Let $\lambda:= d^2/(1-2d)^2$. Let $x_0>0$ and  $(x_0,y_0) \in \Omega$. We denote $ (x_2,y_2) =T_d^2 (x_0,y_0) $. Then, 
	\[
	0\le x_2 < \lambda \, x_0.
	\] }

Given $x_0>0$, let 
$L := \{x=x_0\} \cap D_2 
= \{ (x_0,t)\mid \ t_1= t_1(x_0) \le t \le t_2(x_0) = t_2 \},$
where $ (x_0,t_1)\in \sigma ^-$ and $ (x_0,t_2)\in \sigma ^+$ and hence 
$$
t_1=(x_0+t_1)^d \qquad \text {and} \qquad t_2=2(x_0+t_2)^d+1.
$$ 
  
The image of $L$ by $T_d$ can be represented by
\[
\Gamma_1 (t) = \left(\begin{array}{c}\alpha(t) \\
	\beta(t)\end{array}\right)
:=T_d\left(\begin{array}{c}x_0 \\
	t\end{array}\right) =
\left(\begin{array}{c}t-(x_0+t)^d \\
	t-2(x_0+t)^d\end{array}\right), \qquad t\in[t_1,t_2].
\]
  
Since $d-1$ is even  and the fact that if $(x,y) \in D_2$  we have $y>2(x+y)^d+1$ and $x+y<0$,  
\begin{equation}\label{eq:alpha_prime}
	\alpha^{\prime}(t)=1-d\left(x_0+t\right)^{d-1}\le 1-d\left(\frac{t-1}{2}\right)^{(d-1)/d}\le 1-d<0.
\end{equation}
  
This means that $\alpha(t)$ is strictly decreasing in $t$ with
\begin{equation}\label{fitesalpha}
\alpha(t_2)=(t_2+1)/2 \le \alpha(t)\le   \alpha(t_1)=0.
\end{equation}
  
Similarly we have $\beta^{\prime}(t)=1-2d\left(x_0+t\right)^{d-1}
\le 1-2d<0$ and   
$\beta(t_2)=1 \le \beta(t)\le   \beta(t_1)=-t_1.$

This implies that $\Gamma_1(t)$ can be seen as the graph of an increasing function joining $\J_1$ with $\J_2$. Therefore, it crosses transversally the boundary of $E_2$. Let
$\left(x^{+},y^{+}\right)\in\gamma^+$ and  $\left(x^{-},y^{-}\right)\in\gamma^{-}$ be the corresponding  intersections. 
From \eqref{fitesalpha} and \eqref{eq:alpha_prime} we have 
\begin{equation}\label{eq:t_2}
(t_2+1)/2 \le x^- \le x \qquad \text{for all} \quad (x,y) \in \Gamma_1 \cap E_2.	
\end{equation}
  
Now, given $\xi\in[x^-,0]$ we consider the new vertical segment in $E_2$, 
\[
L_2:=\{x=\xi \}\cap E_2.
\]
By its definition $T_d(L_2)$ is a curve joining 
$\mathcal I_1$ and $\mathcal I_2$, parametrized by
$$
\Gamma_2(s)=\left(\begin{array}{c}\alpha_2(s) \\
\beta_2(s)\end{array}\right)
:=
T_d\left(\begin{array}{c}\xi \\
	s\end{array}\right) 
=
\left(\begin{array}{c}s-(\xi +s)^d \\
	s-2(\xi +s)^d\end{array}\right), \qquad s\in[s_1,s_2],
$$
where $s_1=s_1(\xi)$, $s_2=s_2(\xi)$ and
$$
s_1=2(\xi +s_1)^d-1>1 \qquad \text {and} \qquad s_2=(\xi +s_2)^d.
$$ 
  
A similar computation to the one in \eqref{eq:alpha_prime} gives that 
$\beta_2^{\prime}(s) < \alpha_2^{\prime}(s)\le 1-d<0
$ and then
$$
\alpha_2(s_2)
\le \alpha_2(s)\le \alpha_2(s_1)
=\frac{s_1-1}{2}, \qquad s\in[s_1,s_2].
$$
The claim will follow from $\alpha_2(s_1(\xi))=\frac{s_1(\xi)-1}{2} \le \lambda x_0$
for all $\xi\in[x^-,0]$. 

Clearly $\sigma^+$ is above its tangent line at  the point $(0,-1)$ which is given by $y=\frac{2d}{1-2d}x-1$ (to get the slope of the line  we can use implicit derivation to 
$y=1+2(x+y)^d$ at $(x,y)=(0,-1)$). As a consequence,
since $(x_0,t_2)\in \sigma^+$, 
\begin{equation}\label{eq:t_2_bis}
	t_2+1>\frac{2d}{1-2d}x_0.
\end{equation}
Similarly, $\gamma^{-}$ is below its tangent line at  the point $(0,1)$ given by $y=\frac{2d}{1-2d}x+1$.  Consequently, since $(\xi,s_1(\xi))\in \gamma^-$, 
\begin{equation}\label{eq:s_1}
	s_1(\xi)-1<\frac{2d}{1-2d}\xi.
\end{equation}
Using \eqref{eq:t_2}, \eqref{eq:t_2_bis} and \eqref{eq:s_1}  we have that 
$$
\alpha_2(s_1(\xi))=\frac{s_1(\xi)-1}{2}<\frac{d}{1-2d}\xi
\le \frac{d}{1-2d}x^-
\le \frac{d}{1-2d} \frac{t_2+1}{2}
\le \left(\frac{d}{1-2d}\right)^2x_0=\lambda \,x_0.
$$

\vglue 0.2truecm

\noindent {\bf Claim 2.} Let $x=x_0>0$. Then, there exists a point $(x_0,y)\in\Omega $ such that $(x_0,y)\in W^s$. In particular, from Proposition \ref{prop:inclusio}, we conclude that $\partial \mathcal A(0)$ is unbounded.

\vglue 0.2truecm

 Since $(0,-1)$ is hyperbolic we already know that $W^s_{p_1}$ exists and consists of the points such that their $\omega$-limit with respect to $T_d^2$ is $(0,-1)$.

We recall that $\Omega\subset T_d^2(\Omega) = Q^\star _4$.
Thus, 
\begin{equation}
	\label{key:Omega} 
 T_d^{-2}(\Omega) \subset \Omega.
\end{equation}

Let $K_0:=\{x=x_0\} \cap \Omega$. Clearly, $T_d(K_0)$ is a curve connecting $\gamma^{-}$ and $\gamma^{+}$ and so, by construction,  $T_d^2\left(K_0\right)$ is a curve connecting $\mathcal I_1$ and $\mathcal I_2$ and crossing $\partial \Omega$ at exactly two points (remember that $T_d$ is one-to-one): one in $T_d^{-1}\left(\gamma^{-}\right)$ and the other in $T_d^{-1}\left(\gamma^{+}\right)$. 
We write $ K_1=T_d^{-2}(T_d^2\left(K_0\right)\cap \Omega)\subset K_0\subset \Omega$.

Repeating this procedure we can define recursively 
\[
K_j=T_d^{-2j}(T_d^{2j} (K_{j-1})\cap \Omega)\subset K_{j-1}, \qquad j\geq 1.
\]
Therefore, $\{K_j\}_{j\geq 0}$ is a sequence of nested compact sets and therefore 
$$
\bigcap_{j\geq 0} K_j \ne \emptyset.
$$  
Now, we check that if $(x_0,y_0) \in \bigcap_{j\geq 0} K_j $, then  
$(x_0,y_0) \in W^s$.
Indeed, let $(x_0,y_0) \in \bigcap_{j\geq 0} K_j $. By the definition of $K_j$, 
\[
(x_{2j},y_{2j})= T_d^{2j}(x_0,y_0) \in T_d^{2j} (K_{j-1})\cap \Omega
\subset  T_d^{2j} (K_{0})\cap \Omega
, \qquad j\ge 0.
\] 
We can prove by induction that $x_{2j} < \la^j x_0$ for all $j\ge 1$. Since 
$(x_0,y_0)\in K_0\cap \Omega$, by Claim  1,  $x_{2} < \la x_0$. Assuming the statement is true for $j-1$, since 
$(x_{2j-2},y_{2j-2}) \in T_d^{2j-2} (K_0) \cap \Omega $, then
$(x_{2j},y_{2j}) = T_d^{2}(x_{2j-2},y_{2j-2}) $ satisfies 
$x_{2j} < \la x_{2j-2}$. We conclude that $ x_{2j} \to 0$. Since 
$(x_{2j},y_{2j}) \in \Omega $ we also have $ y_{2j} \to -1$.
Since $x_0$ is arbitrarily large we obtain that the invariant manifold is unbounded.  
\end{proof}


\section{Proof of Theorem A${\text (c)}$: The case $d$ odd and $a=-1$}\label{sec:d_odd_aminus1}

According to Remark \ref{remark:amenys1}, under the parameter values $d$ odd and $a=-1$, the dynamics on the center manifold of the origin is repelling and therefore the only points tending to the origin under iteration are the ones of the stable manifold of (0,0). Hence, it remains to show that the stable manifold is unbounded. 

It follows from \eqref{eq:T} that for $d$ odd and $a=-1$ the map $T_d$ is a homeomorphism and we have 
\begin{equation}\label{eq:Tpera-1}
	T_d\left(
	\begin{array}{l}
		x   \\ 
		y 
	\end{array}
	\right)  = 
	\left(
	\begin{array}{l}
		y + (x+y)^d \\ 
		y + 2(x+y)^d
	\end{array}
	\right)
	\quad \text{and} \quad 
	T_d^{-1}\left(
	\begin{array}{l}
		x   \\ 
		y 
	\end{array}
	\right)  = 
	\left(
	\begin{array}{l}
		-2x+y - (x-y)^{1/d} \\ 
		 2x-y
	\end{array}
	\right).
\end{equation} 
In a similar way as in the proof of Proposition \ref{prop:d_odd_a_1_unbounded} we introduce a domain, which we expect to contain $W^s$, and we prove that contains points, arbitrarily far away, such that  all their iterates are in the domain and moreover tend to the origin so that indeed it contains $W^s$.  

We will take this domain in the fourth quadrant $Q_4 := \{ x\ge 0,\, y\le 0\}$. We define $D^0\subset Q_4$ by the condition $T_d^2(D^0) = Q_4$. Since $T_d^2$ is a homeomorphism the boundary of $D^0$ is obtained by taking the preimage of the boundary of $Q_4$ with respect to  $T_d^2$.  
  
Consequently, the boundary of $D^0$ is the union of the images of the curves    
\begin{equation}\label{eq:sigma0}
	\begin{aligned}
		\sigma^+_0(t)& := (\al^+_0 (t) , \be^+_0 (t))
		= T_d^{-2}(t,0) =\Big( 
		6t +2t^{1/d}  + (4t+t^{1/d})^{1/d}, -6t -2t^{1/d} \Big),\ \ \ \ t\ge 0,\\
		\sigma^-_0(t) &   := (\al^-_0 (t) , \be^-_0 (t))
		= T_d^{-2}(0,-t) =\Big( 
		3t +2t^{1/d}  + (2t+t^{1/d})^{1/d}, -3t -2t^{1/d} \Big),\ \ t\ge 0.
	\end{aligned}
\end{equation}
  
We have that 
$(\al^\pm_0)' (t)>0  $ and $ (\be^\pm_0 ) ' (t) <0$. Therefore, the curves $\sigma^\pm_0$ are graphs of well defined decreasing functions 
$h^\pm_0 = \beta^\pm_0 \circ (\al ^\pm _0)^{-1}$ from $[0,\infty)
$ onto $(-\infty, 0]$.

By construction, the set $D^1 := T_d(D^0)=T_d^{-1}(Q_0)$ is the domain limited by the curves 
$\sigma ^{\pm }_1 := T_d(\sigma^{\pm }_0(t) ),\ t\ge 0$. 
Concretely, these curves are
\begin{equation}\label{eq:sigma1}
	\begin{aligned}
		\sigma^+_1(t)& =(\al^+_1 (t) , \be^+_1 (t)) = 
		T_d(\sigma^{+ }_0(t)) =
		\big( 
		-2t -t^{1/d} ,2t \big), \quad t\ge 0,	\\
		\sigma^-_1(t) & =(\al^-_1 (t) , \be^-_1 (t))
		= T_d(\sigma^{-}_0(t)) = \big( 
		-t -t^{1/d} , t\big), \ \ \ \quad t\ge 0.
	\end{aligned}
\end{equation}
  
Similarly, since $(\al^\pm_1)' (t)<0 $ and $ (\be^\pm_1 ) ' (t) >0$, we have that $\sigma ^{\pm }_1$ are graphs of decreasing functions 
$h^\pm_1 = \beta^\pm_1 \circ (\al ^\pm _1)^{-1}$ from 
$(-\infty, 0]$ onto $[0,\infty) $. 

Finally, $D^2:= T_d(D^1) $ is the full closed fourth quadrant $Q_4$. For notational convenience we also define 
\begin{equation*}\label{eq:sigma2}
	\begin{aligned}
		\sigma^+_2(t)& :=  
		T_d(\sigma^{+ }_1(t)) =
		\big( t,0 \big), \qquad t\ge 0,	\\
		\sigma^-_2(t) & 
		:= T_d(\sigma^{-}_1(t)) = \big( 
		0 , -t\big), \ \quad t\ge 0.
	\end{aligned}
\end{equation*}
Since $T_d$ is a homeomorphism, the curves $\sigma^+_1$ and $\sigma^-_1$ are the only preimages of the curves $\sigma^+_2$ and $\sigma^-_2$, respectively. They only intersect at the origin. The same happens with $\sigma^+_0$  and $\sigma^-_0$. Moreover, $\sigma^+_1$ is above  $\sigma^-_1$ and  $\sigma^+_0$ is above  $\sigma^-_0$. Also, we will use that  $\sigma^+_1$ is below  $\{y=-x\}$
and  $\sigma^-_0$ is above  $\{y=-x\}$.
Indeed, these claims can be checked from \eqref{eq:sigma0} and \eqref{eq:sigma1} after some computations.

\begin{lemma} \label{fitadividitper3}
	 If $(x_0,y_0) \in D^0$ then $(x_2,y_2):=T_d^2(x_0,y_0) \in Q_4$ and $x_2 \le x_0/2$. 
\end{lemma} 
\begin{proof}
	The first part of the statement follows from the previous construction. 
	We have to prove the inequality. We define
	\begin{equation*}
		D^0_\rho =\{(x,y)\in D^0\mid \ x\le \rho \}, \quad 
		D^1_\rho =\{(x,y)\in D^1\mid \ x\ge -\rho \}, \quad
		D^2_\rho =\{(x,y)\in D^2\mid \ x\le \rho \}.
	\end{equation*}
	  
	We will prove that, for any $\rho>0$,  
	$$
	T_d(D^0_\rho ) \subset D^1_{\rho/2 } \qquad \text{and}\qquad 
	T_d(D^1_{\rho/2 }) \subset D^2_{\rho/2 }.
	$$
	  
For the first inclusion  
we consider the segments 
$\{x=r\} \cap D^0_\rho$ with $0\le r\le \rho$, parametrized by $s \in[s_-, s_+]\subset [-r,0]$,
with $s_-$ and $s_+$ such that $(r,s_-)\in \{\image \sigma_0^-\}$ and  
$(r,s_+)\in \{\image \sigma_0^+\}$. In particular, we have that there exists $t_1\ge 0$ such that 
$$
\sigma^-_0(t_1)=
\Big( 
3t_1 +2t_1^{1/d}  + (2t_1+t_1^{1/d})^{1/d}, -3t_1-2t_1^{1/d} \Big)=(r,s_-).
$$
The image of the segment can be represented by
$$
\tau_1(s) = \big( \tau_1^x(s), \tau_1^y(s) \big)
:= T_d(r,s)
= \big( s+(r+s) ^d, s+2 (r+s) ^d\big), \qquad s \in[s_-, s_+].
$$ 

Since $d-1$ is even, 
$( \tau_1^x)'(s)$ and $ (\tau_1^y)' (s)$ are positive.
This implies that the minimum of  $\tau_1^x(s)$ is attained at the value $s=s_-$.
This point is sent by $T_d$ to 
$
(-t_1- t_1^{1/d}, \, t_1)
$ 
and we have
$$
 -t_1- t_1^{1/d} = \frac12 (-3t_1 - 2 t_1^{1/d}) + \frac12 t_1 
 \ge  \frac12 s_-
\ge -\frac12 r \ge -\frac12 \rho. 
$$

Now let $\wt r$ be such that $ -\rho/2 \le \wt r \le 0$ and we consider the image of the segment 
$\{x=\wt r\} \cap \D^1_\rho$, parametrized by 
$\wt s  \in[\wt s_-, \wt s_+]\subset [0, -\wt r]$.
We write
$$
\tau_2(\wt s) = \big( \tau_2^x(\wt s), \tau_2^y(\wt s) \big) := T_d(\wt r,\wt s)
= \big( \wt s+(\wt r+\wt s) ^d, \wt s+2 (\wt r+\wt s) ^d\big). 
$$   
Since $\sigma^+_1$ is below $\{y=-x\}$, $\wt r+\wt s<0$.
In this case we also have that 
$(\tau_2^x) ' (\wt s)$ and $ (\tau_2^y) '  (\wt s)$ are positive. Then, a bound of the maximum of  $\tau_2^x(\wt s)$ is obtained from
$$
\tau_2^x(\wt s) \le \tau_2^x(-\wt r)=  - \wt r \le \rho/2. 
$$
This implies $T_d(D^1_{\rho/2 }) \subset D^0_{\rho/2 }$.
\end{proof}

Now we take $\rho>0$ arbitrary and define 
$$
I^0 = D^0 \cap \{x=\rho\}.
$$
Its image by $T_d^2$ is a curve in $Q_4 $ that joints a point in $\{\image \sigma _2^-\}$ and a point in $\{\image \sigma _2^+\}$. Then this curve has to cross 
$\{\image \sigma_0^-\}$ and  $\{\image \sigma _0^+\}$.
The set $
I^1 = T_d^{-2} (T_d^2(I^0) \cap D^0) \subset I^0
$
contains points such that they, together with their second iterates, belong to $I^0$. Repeating this procedure we define, as in the final part of the proof of Proposition
\ref{prop:d_odd_a_1_unbounded}, 
$$
I^{k} = T_d^{-2k} (T_d^{2k}(I^{k-1}) \cap D^0) .
$$ 
Clearly,  $I^{k} \subset I^{k-1}$ so that  $I^{k}$ is  a sequence of nested compact sets as well.

Then $I^\infty = \cap_{n\ge 0} I^n \neq \emptyset$.
By this construction, if $(x_0,y_0) \in I^\infty$, 
$(x_{2k},y_{2k}) = T_d^{2k}(x_0,y_0) \in T_d^{2k}(I^{k-1}) \cap D^0 \subset D^0$ for all $k\ge 0$.
Then, by Lemma  \ref{fitadividitper3}, 
$$
0< x_{2k} < \left( \frac12\right)^k x_0
$$ 
and as $(x_{2k},y_{2k})\in D^0$, the iterates 
$(x_{2k},y_{2k})$ converge to $(0,0)$  which implies that $(x_0,y_0) \in W^s_0$. Since $\rho$ is arbitrary,  $W^s_0$ is unbounded.
\begin{remark}
	Since the curves that determine th   e boundary of $D^0$ are very close they provide a very good approximation for the stable manifold, even far away from the origin. 
\end{remark}


\bibliographystyle{alpha}
\bibliography{biblio}

\begin{thebibliography}{CFdlL05}

\bibitem[Ber93]{Ber}
Walter Bergweiler.
\newblock Iteration of meromorphic functions.
\newblock {\em Bull. Amer. Math. Soc. (N.S.)}, 29(2):151--188, 1993.

\bibitem[BF18]{BedFri}
Eric Bedford and Paul Frigge.
\newblock The secant method for root finding, viewed as a dynamical system.
\newblock {\em Dolomites Res. Notes Approx.}, 11(Special Issue Norm
  Levenberg):122--129, 2018.

\bibitem[BFJK18]{BarFagJarKar}
Krzysztof Bara\'{n}ski, N\'{u}ria Fagella, Xavier Jarque, and Bogus{\l}awa
  Karpi\'{n}ska.
\newblock Connectivity of {J}ulia sets of {N}ewton maps: a unified approach.
\newblock {\em Rev. Mat. Iberoam.}, 34(3):1211--1228, 2018.

\bibitem[Bla84]{Bla2}
Paul Blanchard.
\newblock Complex analytic dynamics on the {R}iemann sphere.
\newblock {\em Bull. Amer. Math. Soc. (N.S.)}, 11(1):85--141, 1984.

\bibitem[Bla94]{Bla1}
Paul Blanchard.
\newblock The dynamics of {N}ewton's method.
\newblock In {\em Complex dynamical systems ({C}incinnati, {OH}, 1994)},
  volume~49 of {\em Proc. Sympos. Appl. Math.}, pages 139--154. Amer. Math.
  Soc., Providence, RI, 1994.

\bibitem[CFdlL03]{ParametrizationMethod1}
Xavier Cabr\'{e}, Ernest Fontich, and Rafael de~la Llave.
\newblock The parameterization method for invariant manifolds. {I}. {M}anifolds
  associated to non-resonant subspaces.
\newblock {\em Indiana Univ. Math. J.}, 52(2):283--328, 2003.

\bibitem[CFdlL05]{ParametrizationMethod3}
Xavier Cabr\'{e}, Ernest Fontich, and Rafael de~la Llave.
\newblock The parameterization method for invariant manifolds. {III}.
  {O}verview and applications.
\newblock {\em J. Differential Equations}, 218(2):444--515, 2005.

\bibitem[Evg78]{Evg}
M.~A. Evgrafov.
\newblock {\em Analytic functions}.
\newblock Dover Publications, Inc., New York, 1978.
\newblock Translated from the Russian, Reprint of the 1966 original English
  translation, Edited and with a foreword by Bernard R. Gelbaum.

\bibitem[FGJ24]{FGJ24b}
Ernest Fontich, Antoni Garijo, and Xavier Jarque.
\newblock Chaotic dynamics at the boundary of a basin of attraction via
  non-transversal intersections for a non-global smooth diffeomorphism.
\newblock Preprint, 2024.

\bibitem[GGJ21]{GarGarJar}
Laura Gardini, Antonio Garijo, and Xavier Jarque.
\newblock Topological properties of the immediate basins of attraction for the
  secant method.
\newblock {\em Mediterr. J. Math.}, 18(5):Paper No. 221, 27, 2021.

\bibitem[GJ19]{Tangent}
Antonio Garijo and Xavier Jarque.
\newblock Global dynamics of the real secant method.
\newblock {\em Nonlinearity}, 32(11):4557--4578, 2019.

\bibitem[GJ20]{Multiple}
Antonio Garijo and Xavier Jarque.
\newblock The secant map applied to a real polynomial with multiple roots.
\newblock {\em Discrete Contin. Dyn. Syst.}, 40(12):6783--6794, 2020.

\bibitem[HPS77]{HPS77}
M.~W. Hirsch, C.~C. Pugh, and M.~Shub.
\newblock {\em Invariant manifolds}, volume Vol. 583 of {\em Lecture Notes in
  Mathematics}.
\newblock Springer-Verlag, Berlin-New York, 1977.

\bibitem[HSS01]{HowToNewton}
John Hubbard, Dierk Schleicher, and Scott Sutherland.
\newblock How to find all roots of complex polynomials by {N}ewton's method.
\newblock {\em Invent. Math.}, 146(1):1--33, 2001.

\bibitem[Pal69]{PALIS69}
J.~Palis.
\newblock On {M}orse-{S}male dynamical systems.
\newblock {\em Topology}, 8(4):385--404, 1969.

\bibitem[Shi09]{Shi}
Mitsuhiro Shishikura.
\newblock The connectivity of the {J}ulia set and fixed points.
\newblock In {\em Complex dynamics}, pages 257--276. A K Peters, Wellesley, MA,
  2009.

\end{thebibliography}

\end{document}